\documentclass[11pt, letterpaper]{article}

\usepackage{amssymb,amsmath,amsfonts}
\usepackage{hyperref}
\usepackage{graphicx,color}
\usepackage{calrsfs}		% Caligraphic fonts
\usepackage{amsthm}	% For proof environment
\usepackage[round]{natbib}

\usepackage{comment}

\usepackage{geometry}
\usepackage{enumitem}

\newcommand{\lc}{[\![}
\newcommand{\rc}{]\!]}
\newcommand{\1}[1]{{\boldsymbol 1_{\{#1\}}}}
\newcommand{\oo}{\boldsymbol 1}

\newcommand{\E}{\mathbb{E}}
\newcommand{\F}{\mathbb{F}}
\newcommand{\G}{\mathbb{G}}
\newcommand{\Hb}{\mathbb{H}}
\newcommand{\N}{\mathbb{N}}
\renewcommand{\P}{\mathbb{P}}
\newcommand{\Q}{\mathbb{Q}}
\newcommand{\R}{\mathbb{R}}
\newcommand{\dd}{\mathrm{d}}

\newcommand{\Ecal}{{\mathcal E}}
\newcommand{\Fcal}{{\mathcal F}}
\newcommand{\Gcal}{{\mathcal G}}

\newcommand{\Lcal}{{\mathcal L}}
\newcommand{\Tcal}{{\mathcal T}}

\theoremstyle{plain}
\newtheorem{theorem}{Theorem}
\newtheorem{corollary}[theorem]{Corollary}
\newtheorem{assumption}[theorem]{Assumption}
\newtheorem{lemma}[theorem]{Lemma}

\theoremstyle{definition}
\newtheorem{definition}[theorem]{Definition}

\theoremstyle{definition}
\newtheorem{remark}[theorem]{Remark}
\newtheorem{important remark}[theorem]{Important remark}
\newtheorem{example}[theorem]{Example}

\numberwithin{equation}{section}
\numberwithin{theorem}{section}

\allowdisplaybreaks

\begin{document}

\title{Convergence of local supermartingales and Novikov-Kazamaki type conditions for processes with jumps\thanks{We thank Tilmann Bl\"ummel, Pavel Chigansky, Sam Cohen, Christoph Czichowsky, Freddy Delbaen, Moritz D\"umbgen, Hardy Hulley, Jan Kallsen, Ioannis Karatzas, Kostas Kardaras,  Kasper Larsen,  and Nicolas Perkowski for discussions on the subject matter of this paper. We are indebted to Dominique L\'epingle for sending us the paper \citet{Lepingle_Memin_integrabilite}. M.L.~acknowledges funding from the European Research Council under the European Union's Seventh Framework Programme (FP/2007-2013) / ERC Grant Agreement n.~307465-POLYTE. J.R.~acknowledges generous support from the Oxford-Man Institute of Quantitative Finance, University of Oxford, where a major part of this work was completed.
}}
\author{Martin Larsson\thanks{Department of Mathematics, ETH Zurich, R\"amistrasse 101, CH-8092, Zurich, Switzerland. E-mail: martin.larsson@math.ethz.ch} \and
Johannes Ruf\thanks{Department of Mathematics, University College London, Gower Street, London WC1E 6BT, United Kingdom. E-mail:     j.ruf@ucl.ac.uk}}

\date{November 23, 2014}
%\date{\today}

\maketitle

\begin{abstract}
We characterize the event of convergence of a local supermartingale. Conditions are given in terms of its predictable characteristics and quadratic variation. The notion of extended local integrability plays a key role.  We then apply these characterizations to provide a novel proof for the sufficiency and necessity of Novikov-Kazamaki type conditions for the martingale property of nonnegative local martingales with jumps.


{\bf MSC2010 subject classification:} Primary 60G07; secondary: 60G17, 60G30, 60G44.
\end{abstract}

\section{Introduction}

Among the most fundamental results in the theory of martingales are the martingale and supermartingale convergence theorems of \citet{Doob:1953}. One of Doob's results states that if $X$ is a nonnegative supermartingale, then $\lim_{t\to\infty}X_t$ exists almost surely. If $X$ is not nonnegative, or more generally fails to satisfy suitable integrability conditions, then the limit need not exist, or may only exist with some probability. One is therefore naturally led to search for convenient characterizations of the event of convergence $D=\{\lim_{t\to\infty}X_t \text{ exists in }\R\}$. An archetypical example of such a characterization arises from the Dambis-Dubins-Schwarz theorem: if~$X$ is a continuous local martingale, then $D=\{[X,X]_{\infty-}<\infty\}$ almost surely. This equality fails in general, however, if $X$ is not continuous, in which case it is natural to ask for a description of how the two events differ. The first main goal of the present paper is to address questions of this type: how can one describe the event of convergence of a process $X$, as well as of various related processes of interest? We do this in the setting where $X$ is a {\em local supermartingale on a stochastic interval $\lc0,\tau\lc$}, where $\tau$ is a foretellable time. (Precise definitions are given below, but we remark already here that every predictable time is foretellable.)

While the continuous case is relatively simple, the general case offers a much wider range of phenomena. For instance, there exist locally bounded martingales $X$ for which both $\lim_{t\to \infty}X_t$ exists in $\R$ and $[X,X]_{\infty-}=\infty$, or for which $\liminf_{t\to\infty}X_t=-\infty$, $\limsup_{t\to\infty}X_t=\infty$, and $[X,X]_{\infty-}<\infty$ hold simultaneously almost surely. We provide a large number of examples of this type. To tame this disparate behavior, some form of restriction on the jump sizes is needed. The correct additional property is that of {\em extended local integrability}, which is a modification of the usual notion of local integrability.

Our original motivation for considering questions of convergence came from the study of Novikov-Kazamaki type conditions for a nonnegative local martingale $Z=\Ecal(M)$ to be a uniformly integrable martingale. Here $\Ecal(\cdot)$ denotes the stochastic exponential and $M$ is a local martingale. This problem was originally posed by \citet{Girsanov_1960}, and is of great importance in a variety of applications, for example in mathematical finance, where $Z$ corresponds to the Radon-Nikodym density process of a so-called risk-neutral measure. Early sufficient conditions were obtained by \citet{Gikhman_Skorohod_1972} and \citet{Lip_Shir_1972}. An important milestone is due to \cite{Novikov} who proved that if $M$ is continuous, then $\E[e^{\frac{1}{2}[M,M]_\infty}]<\infty$ implies that $Z$ is a uniformly integrable martingale. \cite{Kazamaki_1977} and \citet{Kazamaki_1983} later proved that $\sup_\sigma \E[e^{\frac{1}{2}M_\sigma}]<\infty$ is in fact sufficient, where the supremum is taken over all bounded stopping times~$\sigma$. These results have been generalized in a number of ways. The general case where $M$ may exhibit jumps has been considered by \cite{Novikov:1975}, \citet{Lepingle_Memin_integrabilite,Lepingle_Memin_Sur}, \citet{Okada_1982}, \citet{Yan_1982}, \citet{Kallsen_Shir}, \citet{Protter_Shimbo},  \citet{Sokol2013_optimal},  and \citet{GlauGrbac_2014}, among others. Approaches related to the one we present here can be found in \cite{Kabanov/Liptser/Shiryaev:1979,Kabanov/Liptser/Shiryaev:1980}, \citet{CFY}, \citet{KMK2010},  \citet{Mayerhofer_2011}, \citet{MU_martingale}, \citet{Ruf_Novikov, Ruf_martingale}, \citet{Blanchet_Ruf_2012}, \citet{Klebaner:2014}, and \citet{Hulley_Ruf:2015}, among others.

Let us indicate how questions of convergence arise naturally in this context, assuming for simplicity that $M$ is continuous and $Z$ strictly positive, which is the situation studied by~\cite{Ruf_Novikov}. For any bounded stopping time $\sigma$ we have
\[
\E_\P\left[ e^{ \frac{1}{2} [M,M]_\sigma } \right] = \E_\P\left[Z_\sigma e^{ -M_\sigma + [M,M]_\sigma} \right].
\]
While {\em a priori} $Z$ need not be a uniformly integrable martingale, one can still find a probability measure~$\Q$, sometimes called the {\em F\"ollmer measure}, under which $Z$ may explode, say at time $\tau_\infty$, and such that $\dd \Q / \dd \P |_{\mathcal F_\sigma} = Z_\sigma$ holds for any bounded stopping time $\sigma<\tau_\infty$. For such stopping times,
\[
\E_\P\left[ e^{ \frac{1}{2} [M,M]_\sigma } \right] = \E_\Q\left[ e^{N_\sigma} \right],
\]
where $N=-M+[M,M]$ is a local $\Q$--martingale on $\lc0,\tau_\infty\lc$. The key point is that  $Z$ is a uniformly integrable martingale under $\P$ if and only if $\Q(\lim_{t\to\tau_\infty}N_t\text{ exists in }\R)=1$. The role of Novikov's condition is to guarantee that the latter holds. In the continuous case there is not much more to say; it is the extension of this methodology to the general jump case that requires more sophisticated convergence criteria for the process $X=N$, as well as for certain related processes. Moreover, the fact that $\tau_\infty$ may {\em a priori} be finite explains why we explicitly allow $X$ to be defined on a stochastic interval when we develop the theory. Our convergence results allow us to give simple and transparent proofs of most Novikov-Kazamaki type conditions that are available in the literature. We are also led to necessary and sufficient conditions of this type, yielding improvements of existing criteria, even in the continuous case.

The rest of the paper is organized as follows. Section~\ref{S:prelim} contains notational conventions and mathematical preliminaries. Section~\ref{S:EL} introduces the notion of extended localization and establishes some general properties. Our main convergence theorems and a number of corollaries are given in Section~\ref{S:convergence}. Section~\ref{S:NK} is devoted to Novikov-Kazamaki type conditions. Section~\ref{S:examp} contains counterexamples illustrating the sharpness of the results obtained in Sections~\ref{S:convergence} and~\ref{S:NK}. Auxiliary material is developed in the appendices: Appendix~\ref{A:SE} reviews stochastic exponentials of semimartingales on stochastic intervals. Appendix~\ref{S:follmer} reviews the F\"ollmer measure associated with a nonnegative local martingales. Appendix~\ref{App:Z} characterizes extended local integrability under the F\"ollmer measure. Finally, Appendix~\ref{app:embed} discusses a path space embedding needed to justify our use of the F\"ollmer measure in full generality.

\section{Notation and preliminaries} \label{S:prelim}

In this section we establish some basic notation that will be used throughout the paper. For further details and definitions the reader is referred to \citet{JacodS}.

We work on a stochastic basis $(\Omega,\mathcal F, \mathbb F, \P)$ where $\mathbb F=(\Fcal_t)_{t\ge0}$ is a right-continuous filtration, not necessarily augmented by the $\P$--nullsets.  Given a c\`adl\`ag process $X=(X_t)_{t\ge0}$ we write $X_-$ for its left limits and $\Delta X=X-X_-$ for its jump process, using the convention $X_{0-}=X_0$.  The jump measure of $X$ is denoted by $\mu^X$, and its compensator by $\nu^X$. We let $X^\tau$ denote the process stopped at a stopping time~$\tau$. If $X$ is a semimartingale, $X^c$ denotes its continuous local martingale part, and $H\cdot X$ is the stochastic integral of an $X$--integrable process $H$ with respect to $X$. The stochastic integral of a predictable function $F$ with respect to a random measure $\mu$ is denoted $F*\mu$. For two stopping times $\sigma$ and $\tau$, the stochastic interval $\lc\sigma,\tau\lc$ is the set
\[
\lc\sigma,\tau\lc = \{ (\omega,t)\in\Omega\times\R_+ : \sigma(\omega)\le t<\tau(\omega)\}.
\]
Stochastic intervals such as $\rc\sigma,\tau\rc$ are defined analogously. Note that all stochastic intervals are disjoint from $\Omega\times\{\infty\}$.

A process on a stochastic interval $\lc0,\tau\lc$, where $\tau$ is a stopping time, is a measurable map $X:\lc0,\tau\lc\to\overline\R$. If $\tau'\le\tau$ is another stopping time, we may view $X$ as a process on $\lc0,\tau'\lc$ by considering its restriction to that set; this is often done without explicit mentioning. We say that $X$ is optional (predictable, progressive) if it is the restriction to $\lc0,\tau\lc$ of an optional (predictable, progressive) process. In this paper, $\tau$ will be a foretellable time; that is, a $[0,\infty]$--valued stopping time that admits a nondecreasing sequence $(\tau_n)_{n \in \N}$ of stopping times with $\tau_n<\tau$ almost surely for all $n \in \N$ on the event $\{\tau>0\}$ and $\lim_{n \to \infty} \tau_n = \tau$ almost surely. Such a sequence is called an announcing sequence. We view the stopped process $X^{\tau_n}$ as a process on $\lc0,\infty\lc$ by setting $X_t=X_{\tau_n}\1{\tau_n<\tau}$ for all $t\ge\tau_n$.

If $\tau$ is a foretellable time and $X$ is a process on $\lc0,\tau\lc$, we say that $X$ is a semimartingale on $\lc0,\tau\lc$ if there exists an announcing sequence $(\tau_n)_{n\in \N}$ for $\tau$ such that $X^{\tau_n}$ is a semimartingale for each $n \in \N$. Local martingales and local supermartingales on $\lc0,\tau\lc$ are defined analogously. Basic notions for semimartingales carry over by localization to semimartingales on stochastic intervals. For instance, if $X$ is a semimartingale on $\lc0,\tau\lc$, its quadratic variation process $[X,X]$ is defined as the process on $\lc0,\tau\lc$ that satisfies $[X,X]^{\tau_n} = [X^{\tau_n}, X^{\tau_n}]$ for each $n \in \N$. Its jump measure~$\mu^X$ and compensator~$\nu^X$ are defined analogously, as are stochastic integrals with respect to $X$ (or $\mu^X$, $\nu^X$, $\mu^X-\nu^X$). In particular, $H$ is called $X$--integrable if it is $X^{\tau_n}$--integrable for each $n\in\N$, and $H\cdot X$ is defined as the semimartingale on $\lc 0, \tau\lc$ that satisfies $(H \cdot X)^{\tau_n} = H \cdot X^{\tau_n}$ for each $n \in \N$. Similarly, $G_{\rm loc}(\mu^X)$ denotes the set of predictable functions $F$ for which the compensated integral $F*(\mu^{X^{\tau_n}}-\nu^{X^{\tau_n}})$ is defined for each $n\in\N$ (see Definition~II.1.27 in~\cite{JacodS}), and $F*(\mu^X-\nu^X)$ is the semimartingale on $\lc0,\tau\lc$ that satisfies $(F*(\mu^X-\nu^X))^{\tau_n}=F*(\mu^{X^{\tau_n}}-\nu^{X^{\tau_n}})$ for all $n\in\N$. One easily verifies that all these notions are independent of the particular sequence $(\tau_n)_{n\in\N}$. We refer to \citet{Maisonneuve1977}, \citet{Jacod_book}, and Appendix~A in~\citet{CFR2011} for further details on local martingales on stochastic intervals.

Since we do not require $\Fcal$ to contain all $\P$--nullsets, we may run into measurability problems with quantities like $\sup_{t<\tau}X_t$ for an optional (predictable, progressive) process $X$ on $\lc0,\tau\lc$. However, the left-continuous process $\sup_{t<\cdot}X_t$ is adapted to the $\P$--augmentation $\overline\F$ of $\F$; see the proof of Theorem~IV.33 in~\cite{Dellacherie/Meyer:1978}. Hence it is $\overline \F$--predictable, so we can find an $\F$--predictable process $U$ that is indistinguishable from it; see Lemma~7 in Appendix~1 of~\cite{Dellacherie/Meyer:1982}. Thus the process $V=U\vee X$ is $\F$-optional (predictable, progressive) and indistinguishable from $\sup_{t\le\cdot}X_t$. When writing the latter, we always refer to the indistinguishable process $V$.

We define the set
\[
\Tcal = \{ \tau : \text{ $\tau$ is a bounded stopping time} \}.
\]
Finally, we emphasize the convention $Y(\omega)\oo_A(\omega)=0$ for all (possibly infinite-valued) random variables~$Y$, events $A\in\Fcal$, and $\omega\in\Omega\setminus A$.

\section{The notion of extended localization} \label{S:EL}
The following strengthening of the notion of local integrability and boundedness turns out to be very useful. It is a mild variation of the notion of $\gamma$-localization introduced by~\cite{CS:2005}, see also \cite{Stricker_1981}.

\begin{definition}[Extended locally integrable / bounded]  \label{D:extended}
Let $\tau$ be a foretellable time and $X$ a progressive process on $\lc0,\tau\lc$. Let $D\in\Fcal$. We call $X$ {\em extended locally integrable on $D$} if there exists a nondecreasing sequence $(\tau_n)_{n\in\N}$ of stopping times as well as a sequence $(\Theta_n)_{n \in \N}$ of integrable random variables such that the following two conditions hold almost surely:\begin{enumerate}
\item $\sup_{t\ge0} |X^{\tau_n}_t|\le\Theta_n$ for each $n\in\N$.
\item\label{D:extended:ii} $D\subset\bigcup_{n\in\N}\{\tau_n \geq \tau\}$.
\end{enumerate}
If $D=\Omega$, we simply say that $X$ is {\em extended locally integrable}. Similarly, we call $X$ {\em extended locally bounded (on $D$)} if $\Theta_n$ can be taken deterministic for each $n \in \N$. \qed
\end{definition}

Extended localization naturally suggests itself when one deals with questions of convergence. The reason is the simple inclusion $D\subset\bigcup_{n\in\N}\{X=X^{\tau_n}\}$, where $D$ and $(\tau_n)_{n \in \N}$ are as in Definition~\ref{D:extended}. This inclusion shows that to prove that $X$ converges on $D$, it suffices to prove that each $X^{\tau_n}$ converges on $D$. If $X$ is extended locally integrable on $D$, one may thus assume when proving such results that $X$ is in fact uniformly bounded by an integrable random variable. This extended localization procedure will be used repeatedly throughout the paper.

It is clear that a process is extended locally integrable if it is extended locally bounded.  We now provide some further observations on this strenghtened notion of localization.

\begin{lemma}[Properties of extended localization]  \label{L:ELI}
Let $\tau$ be a foretellable time, $D\in\Fcal$, and $X$ a process on $\lc0,\tau\lc$.
\begin{enumerate}
\item\label{L:ELI:1} If $X=X'+X''$, where $X'$ and $X''$ are extended locally integrable (bounded) on $D$, then $X$ is extended locally integrable (bounded) on $D$.
\item\label{L:ELI:2} If there exists a nondecreasing sequence $(\tau_n)_{n \in \N}$ of stopping times with $D\subset\bigcup_{n\in\N}\{\tau_n \geq \tau\}$ such that $X^{\tau_n}$ is extended locally integrable (bounded) on $D$ for each $n \in \N$, then $X$ is extended locally integrable (bounded) on $D$. In words, an extended locally extended locally integrable (bounded) process is extended locally integrable (bounded).
\item\label{L:ELI:3} Suppose $X$ is c\`adl\`ag adapted. Then $\sup_{t<\tau}|X_t|<\infty$ on $D$ and $\Delta X$ is extended locally integrable (bounded) on $D$ if and only if $X$ is extended locally integrable (bounded) on $D$.
\item\label{L:ELI:5} Suppose $X$ is c\`adl\`ag adapted. Then $x \oo_{x > 1} * \mu^X$ is extended locally integrable on $D$ if and only if $x \oo_{x > 1} * \nu^X$ is extended locally integrable on $D$.  Any of these two conditions imply that $(\Delta X)^+$ is extended locally integrable on $D$.
\item\label{L:ELI:6} Suppose $X$ is optional. If $\sup_{\sigma\in\Tcal}\E[|X_\sigma|\1{\sigma<\tau}]<\infty$ then $X$ is extended locally integrable.
\item\label{L:ELI:4} Suppose $X$ is predictable. Then  $\sup_{t<\tau}|X_t|<\infty$ on $D$ if and only if $X$ is extended locally bounded on $D$ if and only if $X$ is extended locally integrable on $D$.
\end{enumerate}
\end{lemma}

\begin{proof}
The statement in \ref{L:ELI:1} follows by defining a sequence $(\tau_n)_{n \in \N}$ of stopping times by $\tau_n = \tau_n' \wedge \tau_n''$, where $(\tau_n')_{n \in \N}$ and $(\tau_n'')_{n \in \N}$ localize $X'$ and $X''$ extendedly.  For \ref{L:ELI:2}, suppose without loss of generality that $\tau_n\le\tau$ for all $n\in\N$, and let $(\tau_m^{(n)})_{m \in \N}$ localize $X^{\tau_n}$ extendedly, for each $n \in \N$.  Let $m_n$ be the smallest index such that $\P(D \cap \{\tau_{m_n}^{(n)} < \tau_n\}) \leq 2^{-n}$ for each $n \in \N$. Next, define  $\widehat\tau_0=0$ and then iteratively $\widehat\tau_n = \tau_n \wedge (\tau_{m_n}^{(n)} \vee \widehat\tau_{n-1})$ for each $n \in \N$, and check, by applying Borel-Cantelli, that the sequence $(\widehat \tau_n)_{n \in \N}$ satisfies the conditions of Definition~\ref{D:extended}.

For \ref{L:ELI:3} note that the sequence $(\tau_n)_{n \in \N}$ of crossing times, given by $\tau_n=\inf\{t \geq 0:|X_t|\ge n\}$, satisfies Definition~\ref{D:extended}\ref{D:extended:ii}. Thus, by~\ref{L:ELI:2}, it suffices to prove the statement with $X$ replaced by $X^{\tau_n}$. The equivalence then follows directly from the inequalities $|X^{\tau_n}|\le n+|\Delta X_{\tau_n}|\1{\tau_n<\tau}$ and $|\Delta X^{\tau_n}|\le 2 n+|X^{\tau_n}|$

To prove~\ref{L:ELI:5}, suppose first $x\oo_{x>1}*\mu^X$ is extended locally integrable on~$D$. In view of~\ref{L:ELI:2} we may assume by localization that it is dominated by some integrable random variable~$\Theta$, which then yields $\E[x\oo_{x>1}*\nu^X_{\tau-}]\le\E[\Theta]<\infty$. Thus $x\oo_{x>1}*\nu^X$ is dominated by the integrable random variable $x\oo_{x>1}*\nu^X_{\tau-}$, as required. For the converse direction simply interchange $\mu^X$ and $\nu^X$. The fact that $(\Delta X)^+ \leq 1 + x \oo_{x > 1} * \mu^X$ then allows us to conclude.

We now prove \ref{L:ELI:6}, supposing without loss of generality that $X\ge0$. Let $\overline\Fcal$ be the $\P$-completion of $\Fcal$, and write $\P$ also for its extension to $\overline\Fcal$. Define $C=\{\sup_{t<\tau}X_t=\infty\}\in\overline\Fcal$. We first show that $\P(C)=0$, and assume for contradiction that $\P(C)>0$. For each $n\in\N$ define the optional set $O_n=\{t<\tau \text{ and } X_t\ge n\}\subset\Omega\times\R_+$. Then $C=\bigcap_{n\in\N}\pi(O_n)$, where $\pi(O_n)\in\overline \Fcal$ is the projection of $O_n$ onto $\Omega$. The optional section theorem, see Theorem~IV.84 in~\cite{Dellacherie/Meyer:1978}, implies that for each $n\in\N$ there exists a stopping time $\sigma_n$ such that
\begin{equation} \label{eq:L:ELI:section}
\lc\sigma_n\rc\subset O_n \qquad\text{and}\qquad \P\left( \{\sigma_n=\infty\}\cap\pi(O_n) \right) \le \frac{1}{2}\,\P(C).
\end{equation}
Note that the first condition means that $\sigma_n<\tau$ and $X_{\sigma_n}\ge n$ on $\{\sigma_n<\infty\}$ for each $n \in \N$. Thus,
\[
\E[X_{m\wedge\sigma_n}\1{m\wedge\sigma_n<\tau}]
\ge\E[X_{\sigma_n}\oo_{\{\sigma_n\le m\}\cap C}]
\ge n\P(\{\sigma_n<m\}\cap C) \to n\P(\{\sigma_n<\infty\}\cap C)
\]
as $m\to\infty$ for each $n \in \N$. By hypothesis, the left-hand side is bounded by a constant $\kappa$ that does not depend on~$m \in \N$ or~$n \in \N$. Hence, using  that $C\subset\pi(O_n)$ for each $n \in \N$ as well as \eqref{eq:L:ELI:section}, we get
\[
\kappa \ge n\P(\{\sigma_n<\infty\}\cap C) \ge n\Big( \P(C) - \P(\{\sigma_n=\infty\}\cap \pi(O_n))\Big) \ge \frac{n}{2}\,\P(C).
\]
Letting $n$ tend to infinity, this yields a contradiction, proving $\P(C)=0$ as desired. Now define $\tau_n=\inf\{t \geq 0 :X_t\ge n\}$ for each $n\in\N$. By what we just proved, $\P(\bigcup_{n\in\N}\{\tau_n\ge\tau\})=1$. Furthermore, for each $n\in\N$ we have $0\le X^{\tau_n}\le n+X_{\tau_n}\1{\tau_n<\tau}$, which is integrable by assumption. Thus $X$ is extended locally integrable.

For \ref{L:ELI:4}, let $U=\sup_{t<\cdot}|X_t|$. It is clear that extended local boundedness on $D$ implies extended local integrability on $D$ implies $U_{\tau-}<\infty$ on $D$. Hence it suffices to prove that $U_{\tau-}<\infty$ on $D$ implies extended local boundedness on $D$. To this end, we may assume that $\tau < \infty$, possibly after a change of time. We now define a process $U'$ on $\lc 0, \infty \lc$ by $U' = U \oo_{\lc 0, \tau\lc} + U_{\tau-} \oo_{\lc  \tau, \infty\lc} $, and follow the proof of Lemma~I.3.10 in \citet{JacodS} to conclude.
\end{proof}

We do not know whether the implication in Lemma~\ref{L:ELI}\ref{L:ELI:6} holds if $X$ is not assumed to be optional but only progressive.

\begin{example}
	If $X$ is a uniformly integrable martingale then $X$ is extended locally integrable.  This can be seen by considering first crossing times of $|X|$, as in the proof of Lemma~\ref{L:ELI}\ref{L:ELI:3}. \qed 
\end{example}

\section{Convergence of local supermartingales} \label{S:convergence}

In this section we state and prove a number of theorems regarding the event of convergence of a local supermartingale on a stochastic interval. The results are stated in Subsections~\ref{S:conv statements} and~\ref{S:conv locmg}, while the remaining subsections contain the proofs.

\subsection{Convergence results in the general case} \label{S:conv statements}

Our general convergence results will be obtained under the following basic assumption.

\begin{assumption} \label{A:1}
Let $\tau>0$ be a foretellable time with announcing sequence  $(\tau_n)_{n \in \N}$ and $X = M-A$ a local supermartingale on~$\lc 0,\tau\lc$, where $M$ and $A$ are a local martingale and a nondecreasing predictable process on $\lc 0,\tau\lc$, respectively, both starting at zero. 
\end{assumption}

\begin{theorem}[Characterization of the event of convergence] \label{T:conv}
Suppose Assumption~\ref{A:1} holds and fix $D\in\Fcal$. The following conditions are equivalent:
\begin{enumerate}[label={\rm(\alph{*})}, ref={\rm(\alph{*})}]
\item\label{T:conv:a} $\lim_{t\to\tau}X_t$ exists in $\R$ on $D$ and $(\Delta X)^- \wedge X^-$ is extended locally integrable on $D$.
\item\label{T:conv:a'} $\liminf_{t\to\tau}X_t>-\infty$  on $D$ and $(\Delta X)^- \wedge X^-$ is extended locally integrable on $D$.
\item\label{T:conv:b'} $X^-$  is extended locally integrable on $D$.
\item\label{T:conv:b''} $X^+$  is extended locally integrable on $D$ and $A_{\tau-} < \infty$ on $D$.
\item\label{T:conv:c} $[X^c,X^c]_{\tau-} + (x^2\wedge|x|)*\nu^X_{\tau-} + A_{\tau-} <\infty$ on $D$.
\item\label{T:conv:f} $[X,X]_{\tau-}< \infty$ on $D$, $\limsup_{t\to\tau}X_t>-\infty$  on $D$, and $(\Delta X)^- \wedge X^-$ is extended locally integrable on $D$.
\end{enumerate}
If additionally  $X$ is constant after $\tau_J=\inf\{t \geq 0:\Delta X_t=-1\}$, the above conditions are equivalent to the following condition:
	\begin{enumerate}[resume, label={\rm(\alph{*})}, ref={\rm(\alph{*})}]
		\item\label{T:conv:g} Either $\lim_{t\to\tau}\Ecal(X)_t \text{ exists in } \R\setminus\{0\}$ or  $\tau_J < \tau$ on $D$, and $(\Delta X)^- \wedge X^-$ is extended locally integrable on $D$.
	\end{enumerate}
\end{theorem}

\begin{remark} \label{R:4.3}
We make the following observations concerning Theorem~\ref{T:conv}. As in the theorem, we suppose Assumption~\ref{A:1} holds and fix $D\in\Fcal$:
\begin{itemize}
\item For any local supermartingale $X$, the jump process $\Delta X$ is locally integrable. This is however not enough to obtain good convergence theorems as the examples in Section~\ref{S:examp} show. The crucial additional assumption is that localization is in the extended sense. In Subsections~\ref{A:SS:lack} and \ref{A:SS:one}, several examples are collected that illustrate that the conditions of Therorem~\ref{T:conv} are non-redundant, in the sense that the implications fail for some local supermartingale $X$ if some of the conditions is omitted.
\item In \ref{T:conv:c}, we may replace $x^2\oo_{|x|\le \kappa}*\nu^X_{\tau-}$ by $x^2\oo_{|x|\le \kappa}*\mu^X_{\tau-}$, where $\kappa$ is any predictable extended locally integrable process. This follows from a localization argument and Lemma~\ref{L:BJ} below.
\item If any of the conditions \ref{T:conv:a}--\ref{T:conv:f} holds then $\Delta X$ is extended locally integrable on $D$. This is a by-product of the proof of the theorem. The extended local integrability of $\Delta X$ also follows, {\em a posteriori}, from Lemma~\ref{L:ELI}\ref{L:ELI:3} since \ref{T:conv:b'} \& \ref{T:conv:b''} imply that $X$ is extended locally integrable on~$D$.
\item If any of the conditions \ref{T:conv:a}--\ref{T:conv:f} holds and if  $X = M^\prime - A^\prime$ for some local supermartingale $M^\prime$ and some nondecreasing (not necessarily predictable) process $A'$ with $A'_0 = 0$, then $\lim_{t\to\tau} M^\prime_t$ exists in~$\R$ on $D$ and $A^\prime_{\tau-} < \infty$ on $D$. Indeed, $M^\prime \geq X$ and thus the implication {\ref{T:conv:b'}} $\Longrightarrow$  {\ref{T:conv:a}} applied to $M^\prime$ yields that $\lim_{t \to \tau} M^\prime_t$ exists in $\R$, and therefore also $A_{\tau-}^\prime<\infty$.
\item One might conjecture that Theorem~\ref{T:conv} can be generalized to special semimartingales $X=M+A$ on~$\lc0,\tau\lc$ by replacing $A$ with its total variation process ${\rm Var}(A)$ in {\ref{T:conv:b''}} and {\ref{T:conv:c}}. However, such a generalization is not possible in general. As an illustration of what can go wrong, consider the deterministic finite variation process $X_t=A_t=\sum_{n=1}^{[t]}(-1)^n n^{-1}$, where $[t]$ denotes the largest integer less than or equal to~$t$. Then $\lim_{t\to \infty}X_t$ exists in~$\R$, being an alternating series whose terms converge to zero. Thus {\ref{T:conv:a}}--{\ref{T:conv:b'}} \& \ref{T:conv:f} hold with $D=\Omega$. However, the total variation ${\rm Var}(A)_\infty=\sum_{n=1}^\infty n^{-1}$ is infinite, so {\ref{T:conv:b''}} \& {\ref{T:conv:c}} do not hold with $A$ replaced by ${\rm Var}(A)$. Related questions are addressed by~\cite{CS:2005}.
\item One may similarly ask about convergence of local martingales of the form $X=x*(\mu-\nu)$ for some integer-valued random measure $\mu$ with compensator $\nu$. Here nothing can be said in general in terms of $\mu$ and $\nu$; for instance, if $\mu$ is already predictable then $X=0$.
\qed
\end{itemize}
\end{remark}

Therorem~\ref{T:conv} is stated in a  general form and its power appears when one considers specific events $D \in \mathcal F$.  For example, we may let $D = \{\lim_{t\to\tau}X_t \text{ exists in }\R\}$ or $D = \{\liminf_{t\to\tau}X_t>-\infty\}$.  Choices of this kind lead directly to the following corollary.

\begin{corollary}[Extended local integrability from below] \label{C:conv2}
Suppose Assumption~\ref{A:1} holds and $(\Delta X)^- \wedge X^-$ is extended locally integrable on $\{ \text{$\limsup_{t\to\tau}X_t>-\infty$} \}$. Then the following events are almost surely equal:
 \begin{align}
\label{T:conv2:1}
&\Big\{ \text{$\lim_{t\to\tau}X_t$ exists in $\R$} \Big\}; \\
&\Big\{ \text{$\liminf_{t\to\tau}X_t>-\infty$};  \Big\};\label{T:conv2:6}\\
\label{T:conv2:2}
&\Big\{ \text{$[X^c,X^c]_{\tau-} + (x^2 \wedge |x|) * \nu_{\tau-}^X + A_{\tau-} <\infty$} \Big\}; \\
&\Big\{ \text{$[X,X]_{\tau-} <\infty$} \Big\} \bigcap \Big\{\text{$\limsup_{t\to\tau}X_t > -\infty$} \Big\}.\label{T:conv2:5}
\end{align}
\end{corollary}
\begin{proof}
The statement follows directly from Therorem~\ref{T:conv}, where for each inclusion the appropriate event $D$ is fixed.
\end{proof}

We remark that the identity \eqref{T:conv2:1} $=$ \eqref{T:conv2:6} appears already in Theorem~5.19 of \citet{Jacod_book} under slightly more restrictive assumptions, along with the equality
\begin{equation} \label{eq:XMA}
	\Big\{ \text{$\lim_{t\to\tau}X_t$ exists in $\R$} \Big\} = \Big\{ \text{$\lim_{t\to\tau}M_t$ exists in $\R$} \Big\} \bigcap \Big\{ A_{\tau-} < \infty \Big\}.
\end{equation}
Corollary~\ref{C:conv2} yields that this equality in fact holds under assumptions strictly weaker than in \citet{Jacod_book}. Note, however, that some assumption is needed; see Example~\ref{ex:semimartingale}. Furthermore, a special case of the equivalence  \ref{T:conv:f}  $\Longleftrightarrow$  \ref{T:conv:g}  in Therorem~\ref{T:conv} appears in Proposition~1.5 of \citet{Lepingle_Memin_Sur}. Moreover, under additional integrability assumptions on the jumps, Section~4 in \cite{Kabanov/Liptser/Shiryaev:1979} provides related convergence conditions. In general, however, we could not find any of the implications  in Therorem~\ref{T:conv}---except, of course, the trivial implication \ref{T:conv:a} $\Longrightarrow$ \ref{T:conv:a'}---in this generality in the literature. Some of the implications are easy to prove, some of them are more involved. Some of these implications were expected, while others were surprising to us; for example, the limit superior in \ref{T:conv:f} is needed even if~$A=0$ so that~$X$ is a local martingale on $\lc0,\tau\lc$. Of course, whenever the extended local integrability condition appears, then, somewhere in the corresponding proof, so does a reference to the classical supermartingale convergence theorem, which relies on Doob's upcrossing inequality.

\begin{corollary}[Extended local integrability]   \label{C:convergence_QV}
Under Assumption~\ref{A:1}, if $|X| \wedge \Delta X$ is extended locally integrable we have, almost surely,
\begin{align*}
\Big\{ \text{$\lim_{t\to\tau}X_t$ exists in $\R$} \Big\} = \Big\{ \text{$[X,X]_{\tau-} <\infty$} \Big\} \bigcap \Big\{ A_{\tau-}<\infty \Big\}.
\end{align*}
\end{corollary}

\begin{proof}
Note that $\{[X,X]_{\tau-}<\infty\}=\{[M,M]_{\tau-}<\infty\}$ on $\{A_{\tau-}<\infty\}$. Thus, in view of~\eqref{eq:XMA}, it suffices to show that $\{ \text{$\lim_{t\to\tau}M_t$ exists in $\R$} \} = \{ \text{$[M,M]_{\tau-} <\infty$}\}$. The inclusion ``$\subset$'' is immediate from \eqref{T:conv2:1} $\subset$ \eqref{T:conv2:5} in Corollary~\ref{C:conv2}. The reverse inclusion follows noting that
\begin{align*}
\Big\{ \text{$[M,M]_{\tau-} <\infty$} \Big\} &= \left(\Big\{ \text{$[M,M]_{\tau-} <\infty$} \Big\} \cap \Big\{ \limsup_{t\to \tau } M_t > -\infty \Big\}\right)\\
&\qquad \cup \left(\Big\{ \text{$[M,M]_{\tau-} <\infty$} \Big\} \cap \Big\{ \limsup_{t\to \tau } M_t = -\infty \Big\} \cap \Big\{ \limsup_{t\to \tau } (-M_t) > -\infty \Big\}\right)
\end{align*}
and applying the inclusion \eqref{T:conv2:5} $\subset$ \eqref{T:conv2:1} once to $M$ and once to $-M$.
\end{proof}

\begin{corollary}[$L^1$--boundedness]    \label{C:conv001}
Suppose Assumption~\ref{A:1} holds, and let $f:\R\to\R_+$ be any nondecreasing function with $f(x)\ge x$ for all sufficiently large $x$. Then the following conditions are equivalent:
\begin{enumerate}[label={\rm(\alph*)},ref={\rm(\alph*)}]
\item\label{T:conv1:a} $\lim_{t\to\tau}X_t$ exists in $\R$  and $(\Delta X)^- \wedge X^-$ is extended locally integrable.
\item\label{T:conv1:c} $A_{\tau-}<\infty$ and for some  extended locally integrable optional process $U$,
\begin{align}  \label{eq:T:conv:exp}
\sup_{\sigma \in \mathcal{T}} \E\left[f(X_\sigma - U_\sigma)  \oo_{\{\sigma<\tau\}} \right] < \infty.
\end{align}
\item\label{T:conv1:d} For some  extended locally integrable optional process $U$, \eqref{eq:T:conv:exp} holds with $x\mapsto f(x)$ replaced by $x\mapsto f(-x)$.
\item\label{C:conv001:e}  The process $\overline X=X\oo_{\lc0,\tau\lc}+(\limsup_{t\to\tau}X_t)\oo_{\lc\tau,\infty\lc}$, extended to $[0,\infty]$ by $\overline X_\infty  = \limsup_{t\to\tau}X_t$,  is a semimartingale on $[0,\infty]$ and $(\Delta X)^- \wedge X^-$ is extended locally integrable.
\item\label{C:conv001:d} The process $\overline X=X\oo_{\lc0,\tau\lc}+(\limsup_{t\to\tau}X_t)\oo_{\lc\tau,\infty\lc}$,  extended to $[0,\infty]$ by $\overline X_\infty  = \limsup_{t\to\tau}X_t$,  is a special semimartingale on $[0,\infty]$.
\end{enumerate}
\end{corollary}

\begin{proof}
{\ref{T:conv1:a}} $\Longrightarrow$ {\ref{T:conv1:c}} \& {\ref{T:conv1:d}}: The implication {\ref{T:conv:a}} $\Longrightarrow$ {\ref{T:conv:b'}} \& {\ref{T:conv:b''}} in Theorem~\ref{T:conv} yields that $X$ is extended locally integrable, so we may simply take $U=X$.

{\ref{T:conv1:c}} $\Longrightarrow$ {\ref{T:conv1:a}}:
We have $f(x)\ge \1{x\ge \kappa}x^+$ for some constant $\kappa \geq 0$ and all $x\in\R$. Hence~\eqref{eq:T:conv:exp} holds with $f(x)$ replaced by $x^+$. Lemma~\ref{L:ELI}\ref{L:ELI:6} then implies that $(X-U)^+$ is extended locally integrable. Since $X^+\le (X-U)^++U^+$, we have $X^+$ is extended locally integrable.
The implication \ref{T:conv:b''} $\Longrightarrow$ \ref{T:conv:a} in Theorem~\ref{T:conv} now yields~\ref{T:conv1:a}.

{\ref{T:conv1:d}} $\Longrightarrow$ {\ref{T:conv1:c}}:
We now have $f(x)\ge \1{x\le -\kappa}x^-$ for some constant $\kappa \geq 0$ and all $x\in\R$, whence as above, $(X-U)^-$ is extended locally integrable. Since $M^-\le(M-U)^-+U^-\le(X-U)^-+U^-$, it follows that $M^-$ is extended locally integrable. The implication \ref{T:conv:b'} $\Longrightarrow$ \ref{T:conv:a}~\&~\ref{T:conv:b''} in Theorem~\ref{T:conv} yields that $\lim_{t\to\tau}M_t$ exists in $\R$ and that $M$ is extended locally integrable. Hence $A=(U-X+M-U)^+\le(X-U)^-+|M|+|U|$ is extended locally integrable, so Lemma~\ref{L:ELI}\ref{L:ELI:4} yields $A_{\tau-}<\infty$. Thus~{\ref{T:conv1:c}} holds.

\ref{T:conv1:a}   $\Longrightarrow$ \ref{C:conv001:e}:
By \eqref{eq:XMA}, $A$ converges. Moreover, since $M \geq X$, we have  $(\Delta M)^-$ is extended locally integrable by Remark~\ref{R:4.3}, say with localizing sequence $(\rho_n)_{n \in \N}$. Now, it is sufficient to prove that $M^{\rho_n}$ is a local martingale on $[0,\infty]$ for each $n \in \N$, which, however, follows from Lemma~\ref{L:SMC} below.

\ref{C:conv001:e} $\Longrightarrow$ \ref{T:conv1:a}:
Obvious.

\ref{T:conv1:a} \& \ref{C:conv001:e} $\Longleftrightarrow$ \ref{C:conv001:d}:
This equivalence follows from Proposition~II.2.29 in~\cite{JacodS}, in conjunction with the equivalence \ref{T:conv:a}  $\Longleftrightarrow$  \ref{T:conv:c} in Theorem~\ref{T:conv}.
\end{proof}

	Examples~\ref{E:P1} and \ref{ex:semimartingale} below illustrate that the integrability condition is needed in order that \ref{C:conv001}\ref{T:conv1:a} imply the semimartingale property of $X$ on the extended axis. These examples also show that the integrability condition in Corollary~\ref{C:conv001}\ref{C:conv001:e} is not redundant.

\begin{remark} \label{R:bed implied} In Corollary~\ref{C:conv001}, convergence implies not only $L^1$--integrability but also boundedness. Indeed, let $g:\R\to\R_+$ be either $x\mapsto f(x)$ or $x\mapsto f(-x)$. If any of the conditions \ref{T:conv1:a}--\ref{T:conv1:d} in Corollary~\ref{C:conv001} holds then there exists an nondecreasing extended locally integrable optional process $U$ such that the family
\begin{align*}
\left(g(X_\sigma - U_\sigma)  \oo_{\{\sigma<\tau\}} \right)_{\sigma \in \mathcal{T}} \qquad \text{is bounded.}
\end{align*}
To see this, note that if {\ref{T:conv1:a}} holds then $X$ is extended locally integrable. If $g$ is $x\mapsto f(x)$, let $U=\sup_{t \leq \cdot} X_t$, whereas if $g$ is $x\mapsto f(-x)$, let $U=\inf_{t \leq \cdot} X_t$. In either case, $U$ is nondecreasing and extended locally integrable and $(g(X_\sigma-U_\sigma))_{\sigma\in\Tcal}$ is bounded.\qed
\end{remark}

With a suitable choice of $f$ and additional requirements on $U$, condition~\eqref{eq:T:conv:exp} has stronger implications for the tail integrability of the compensator $\nu^X$ than can be deduced, for instance, from Theorem~\ref{T:conv} directly. The following result records the useful case where $f$ is an exponential.

\begin{corollary}[Exponential integrability of~$\nu^X$] \label{C:conv1}
Suppose Assumption~\ref{A:1} holds. If $A_{\tau-}<\infty$ and~\eqref{eq:T:conv:exp} holds with some~$U$ that is extended locally bounded and with $f(x)=e^{cx}$ for some $c\ge1$, then
\begin{equation}\label{T:conv:nu}
(e^x-1-x) * \nu^X_{\tau-} < \infty.
\end{equation}
\end{corollary}

\begin{proof}
In view of Lemma~\ref{L:ELI}\ref{L:ELI:4} we may assume by localization that $A=U=0$ and by Jensen's inequality that $c = 1$. Lemma~\ref{L:ELI}\ref{L:ELI:6} then implies that $e^X$ and hence $X^+$ is extended locally integrable. Thus by Theorem~\ref{T:conv}, $\inf_{t<\tau} X_t >-\infty$. It\^o's formula yields
\begin{align*}
e^X &= 1 + e^{X_-} \cdot X + \frac{1}{2} e^{X_-}\cdot [X^c,X^c] + \left(e^{X_-}(e^x-1-x)\right)*\mu^X.
\end{align*}
 The second term on the right-hand side is a local martingale on $\lc0,\tau\lc$, so we may find a localizing sequence $(\rho_n)_{n\in\N}$ with $\rho_n<\tau$. Taking expectations and using the defining property of the compensator~$\nu^X$ as well as the associativity of the stochastic integral yield
\[
\E\left[ e^{X_{\rho_n}} \right] = 1+  \E\left[ e^{X_-}\cdot\left(\frac{1}{2} [X^c,X^c] + (e^x-1-x)*\nu^X\right)_{\rho_n} \right]
\]
for each $n \in \N$.
Due to~\eqref{eq:T:conv:exp}, the left-hand side is bounded by a constant that does not depend on~$n \in \N$. We now let~$n$ tend to infinity and recall that $\inf_{t<\tau} X_t >-\infty$ to deduce by the monotone convergence theorem that~\eqref{T:conv:nu} holds.
\end{proof}

\begin{remark}  \label{R:implication}
Extended local integrability of $U$ is not enough in Corollary~\ref{C:conv1}.  For example, consider an integrable random variable $\Theta$ with $\E[\Theta]= 0$ and $\E[e^\Theta] = \infty$ and the process $X =  \Theta \oo_{\lc 1, \infty\lc}$ under its natural filtration. Then $X$ is a martingale. Now, with $U = - X$, \eqref{eq:T:conv:exp} holds with $f(x)=e^x$, but $(e^x - 1 - x) * \nu^X_{\infty-} = \infty$. \qed
\end{remark}

\subsection{Convergence results with jumps bounded below} \label{S:conv locmg}

We now specialize to the case where $X$ is a local martingale on a stochastic interval with jumps bounded from below. The aim is to study a related process $Y$, which appears naturally in connection with the nonnegative local martingale $\Ecal(X)$. We comment on this connection below.

\begin{assumption} \label{A:2}
Let $\tau$ be a foretellable time, and $X$ a local martingale on~$\lc 0,\tau\lc$ with $\Delta X>-1$.  Suppose further that $(x-\log(1+x))*\nu^X$ is finite-valued, and define
\begin{equation*}
Y=X^c+\log(1+x)*(\mu^X-\nu^X)
\end{equation*}
and
\begin{align*}
	\gamma_t= - \int \log(1+x) \nu^X(\{t\},\dd x)
\end{align*}
for all $t<\tau$. 
\end{assumption}

The significance of the process $Y$ originates with the identity
\begin{align} \label{eq:VV}
\Ecal(X)= e^{Y-V} \quad\text{on $\lc0,\tau\lc$,} \quad \text{where} \quad V=\frac{1}{2}[X^c,X^c]+(x-\log(1+x))*\nu^X.
\end{align}
Thus $Y$ is the local martingale part and $-V$ is the predictable finite variation part of the local supermartingale $\log \Ecal(X)$. The process $V$ is called the {\em exponential compensator} of $Y$, and $Y-V$ is called the {\em logarithmic transform} of $X$. These notions play a central role in~\cite{Kallsen_Shir}.

Observe that the jumps of $Y$ can be expressed as
\begin{equation} \label{eq:DY}
\Delta Y_t = \log(1+\Delta X_t) + \gamma_t
\end{equation}
for all $t<\tau$. Jensen's inequality and the fact that $\nu^X(\{t\},\R)\le 1$ imply that $\gamma\ge 0$. If $X$ is quasi-left continuous, then $\gamma\equiv 0$.

In the spirit of our previous results, we now present a theorem that relates convergence of the processes $X$ and $Y$ to the finiteness of various derived quantities.

\begin{theorem}[Joint convergence of a local martingale and its logarithmic transform] \label{T:convYX}
Suppose Assumption~\ref{A:2} holds, and fix $\eta\in(0,1)$ and $\kappa>0$. Then the following events are almost surely equal:
\begin{align}
 	&\Big\{ \lim_{t\to\tau}X_t \text{ exists in }\R\Big\} \bigcap \Big\{ \lim_{t\to\tau}Y_t \text{ exists in }\R\Big\}   \label{T:convYX:1};\\
	&\Big\{\frac{1}{2} [X^c,X^c]_{\tau-}+ (x - \log(1+x)) * \nu^X_{\tau-} < \infty\Big\} ;     \label{T:convYX:2} \\
 	&\Big\{ \lim_{t\to\tau}X_t \text{ exists in }\R\Big\} \bigcap \Big\{ -\log(1+x)\oo_{x< -\eta}*\nu^X_{\tau-} < \infty\Big\}  \label{T:convYX:3};\\
	&\Big\{ \lim_{t\to\tau}Y_t \text{ exists in }\R\Big\} \bigcap   \Big\{x\oo_{x>\kappa} * \nu^X_{\tau-} < \infty\Big\}    .  \label{T:convYX:4}
\end{align}
\end{theorem}

\begin{lemma} \label{L:convYX}
Suppose Assumption~\ref{A:2} holds. For any event $D \in \mathcal F$ with $x\oo_{x>\kappa} * \nu^X_{\tau-} < \infty$ on $D$ for some $\kappa>0$, the following three statements are equivalent:
\begin{enumerate}[label={\rm(\alph{*})}, ref={\rm(\alph{*})}]
\item\label{T:convYX:a} $\lim_{t\to\tau}Y_t$ exists in $\R$ on $D$. 
\item\label{T:convYX:b'} $Y^-$  is extended locally integrable on $D$.
\item\label{T:convYX:b''}  $Y^+$  is extended locally integrable on $D$.
\end{enumerate}
\end{lemma}

\begin{proof}
The implications follow  from Theorem~\ref{T:conv}.  Only that \ref{T:convYX:a} implies \ref{T:convYX:b'} \& \ref{T:convYX:b''} needs an argument, and it suffices to show that $(\Delta (-Y))^-$ is extended locally integrable on $D$. By \eqref{eq:DY} we have $(\Delta(-Y))^-\le(\Delta X)^++\gamma$; Lemma~\ref{L:ELI}\ref{L:ELI:5} implies that $(\Delta X)^+$ is extended locally integrable; and Lemma~\ref{L:convYX:1} below and Lemma~\ref{L:ELI}\ref{L:ELI:6} imply that $\gamma$ is extended locally bounded on~$D$.
\end{proof}

\begin{corollary}[$L^1$--boundedness] \label{C:convYX2}
Suppose Assumption~\ref{A:2} holds, and fix $c\ne0$, $\eta\in(0,1)$, and $\kappa>0$. The following conditions are equivalent:
\begin{enumerate}[label={\rm(\alph*)},ref={\rm(\alph*)}]
\item \label{T:convYX2:a} $\lim_{t\to\tau}X_t$ exists in $\R$ and $-\log(1+x)\oo_{x< -\eta}*\nu^X_{\tau-} < \infty$.
\item\label{T:convYX2:b} $x\oo_{x>\kappa} * \nu^X_{\tau-}< \infty$ and for some extended locally integrable optional process $U$ on $\lc0,\tau\lc$ we have
\begin{equation}  \label{eq:T:convYX:eYp}
\sup_{\sigma \in \mathcal{T}} \E\left[e^{cY_\sigma - U_\sigma}  \1{\sigma<\tau} \right] < \infty.
\end{equation}
\end{enumerate}
If $c\ge 1$, these conditions are implied by the following:
\begin{enumerate}[label={\rm(\alph*)},ref={\rm(\alph*)},resume]
\item\label{T:convYX2:c} \eqref{eq:T:convYX:eYp} holds for some extended locally bounded optional process $U$ on $\lc0,\tau\lc$.
\end{enumerate}
Finally, the conditions \ref{T:convYX2:a}--\ref{T:convYX2:b} imply that $(e^{cY_\sigma-U_\sigma})_{\sigma\in\Tcal}$ is bounded for some extended locally integrable optional process $U$ on $\lc0,\tau\lc$.
\end{corollary}

\begin{proof}
The equivalence of \ref{T:convYX2:a} and \ref{T:convYX2:b} is obtained from \eqref{T:convYX:3} = \eqref{T:convYX:4} in Theorem~\ref{T:convYX}. Indeed, Corollary~\ref{C:conv001} with $X$ replaced by $Y$ and $f(x)=e^{cx}$, together with Lemma~\ref{L:convYX}, yield that \ref{T:convYX2:b} holds if and only if \eqref{T:convYX:4} has full probability. In order to prove that \ref{T:convYX2:c} implies~\ref{T:convYX2:b}  we assume that~\eqref{eq:T:convYX:eYp} holds with $c\geq 1$ and $U$ extended locally bounded. Corollary~\ref{C:conv1} yields
\[
\left(1-\frac{1}{\kappa}\log (1+\kappa)\right)\, (e^y-1)\oo_{y>\log (1+\kappa)}*\nu^Y_{\tau-}\le(e^y-1-y)*\nu^Y_{\tau-}<\infty,
\]
so by a localization argument using Lemma~\ref{L:ELI}\ref{L:ELI:4} we may assume that $(e^y-1)\oo_{y>\log (1+\kappa)}*\nu^Y_{\tau-}\le \kappa_1$ for some constant $\kappa_1>0$. Now, \eqref{eq:DY} yields
\[
\Delta X \1{\Delta X>\kappa} = \left(e^{\Delta Y - \gamma}-1\right)\1{e^{\Delta Y}>(1+\kappa)e^\gamma} \le (e^{\Delta Y}-1)\1{\Delta Y>\log (1+\kappa)},
\]
whence $\E[x\oo_{x>\kappa}*\nu^X_{\tau-}]\le \E[(e^y-1)\oo_{y>\log (1+\kappa)}*\nu^Y_{\tau-}]\le\kappa_1$. Thus~\ref{T:convYX2:b} holds. The last statement of the corollary follows as in Remark~\ref{R:bed implied} after recalling Lemma~\ref{L:convYX}.
\end{proof}

\subsection{Some auxiliary results}
In this subsection, we collect some observations that will be useful for the proofs of the convergence theorems of the previous subsection.

\begin{lemma}[Supermartingale convergence] \label{L:SMC}
Under Assumption~\ref{A:1}, suppose $\sup_{n \in \N} \E[X^-_{\tau_n}]<\infty$. Then the limit $G=\lim_{t\to\tau}X_t$ exists in~$\R$ and the process $\overline X=X\oo_{\lc0,\tau\lc} + G\oo_{\lc\tau,\infty\lc}$, extended to $[0,\infty]$ by $\overline X_\infty=G$, is a supermartingale on $[0,\infty]$ and extended locally integrable. If, in addition, $X$ is a local martingale on $\lc 0,\tau\lc$ then $\overline X$ is a local martingale.
\end{lemma}

\begin{proof}
Supermartingale convergence implies that $G$ exists; see the proof of Proposition~A.4 in~\citet{CFR2011} for a similar statement. 
 Fatou's lemma, applied as in Theorem~1.3.15 in \citet{KS1}, yields the integrability of $\overline X_\rho$ for each $[0,\infty]$--valued stopping time $\rho$, as well as the supermartingale property of $\overline{X}$.
Now, define a sequence of stopping times $(\rho_m)_{m \in \N}$ by
\[
\rho_m = \inf\{ t\ge 0: |\overline{X}_t| > m \}
\]
and note that $\bigcup_{m\in\N}\{\rho_m=\infty\}=\Omega$.  Thus, $\overline{X}$ is extended locally integrable, with the corresponding sequence $(|\overline{X}_{\rho_m}| + m)_{m \in \N}$ of integrable random variables.

Assume now that $X$ is a local martingale and, without loss of generality, that  $\overline{X}^{\tau_n}$ is a uniformly integrable martingale for each $n \in \N$. 
Fix $m \in \N$ and note that  $\lim_{n\to\infty} \overline{X}_{\rho_m\wedge\tau_n}=\overline{X}_{\rho_m}$. Next, the inequality $|\overline{X}_{\rho_m \wedge \tau_n}| \leq |\overline{X}_{\rho_m}| + m$ for each $n \in \N$ justifies an application of dominated convergence as follows:
\[
E\left[\overline{X}_{\rho_m}\right] = E\left[\lim_{n \to \infty} \overline{X}_{\rho_m \wedge \tau_n}\right]  = \lim_{n \to \infty} E\left[\overline{X}_{\rho_m \wedge \tau_n}\right] = 0.
\]
Hence, $\overline{X}$ is a local martingale, with localizing sequence $(\rho_m)_{m \in \N}$.
\end{proof}

For the proof of the next lemma, we are not allowed to use Corollary~\ref{C:convergence_QV}, as it relies on Theorem~\ref{T:conv}, which we have not yet proved.
\begin{lemma}[Continuous case] \label{L:CC}
Let $X$ be a continuous local martingale on $\lc0,\tau\lc$. If $[X,X]_{\tau-}<\infty$ then the limit $\lim_{t\to\tau}X_t$  exists in $\R$.
\end{lemma}

\begin{proof}
See Exercise~IV.1.48 in \citet{RY}.
\end{proof}

The next lemma will serve as a tool to handle truncated jump measures.

\begin{lemma}[Bounded jumps] \label{L:BJ}
Let $\mu$ be an integer-valued random measure such that $\mu(\R_+ \times [-1,1]^c) = 0$, and let $\nu$ be its compensator. Assume either $x^2*\mu_{\infty-}$ or $x^2*\nu_{\infty-}$ is finite. Then so is the other one, we have $x\in G_{\rm loc}(\mu)$, and the limit $\lim_{t\to\infty}x*(\mu-\nu)_t$  exists in $\R$.
\end{lemma}

\begin{proof}
First, the condition on the support of $\mu$ implies that both $x^2*\mu$ and $x^2*\nu$ have jumps bounded by one. 
Now, let $\rho_n$ be the first time $x^2*\nu$ crosses some fixed level $n \in \N$, and consider the local martingale $F=x^2*\mu-x^2*\nu$. Since $F^{\rho_n}\ge-n-1$, the supermartingale convergence theorem implies that $F^{\rho_n}_{\infty-}$ exists in $\R$, whence $x^2*\mu_{\infty-}=F_{\infty-}+x^2*\nu_{\infty-}$ exists and is finite on $\{\rho_n=\infty\}$. This yields
\begin{equation*}
\left\{ x^2*\nu_{\infty-}<\infty\right\}  \subset\left\{ x^2*\mu_{\infty-}<\infty\right\}.
\end{equation*}
The reverse inclusion is proved by interchanging $\mu$ and $\nu$ in the above argument.

Next, the local boundedness of $x^2*\nu$ implies that $x*(\mu-\nu)$ is well-defined and a local martingale with $\langle x*(\mu-\nu),x*(\mu-\nu)\rangle\le x^2*\nu$; see Theorem~II.1.33 in \citet{JacodS}. Hence, for each $n \in \N$, with $\rho_n$ as above, $x*(\mu-\nu)^{\rho_n}$ is a uniformly integrable martingale and thus convergent. Therefore $x*(\mu-\nu)$ is convergent on the set $\{\rho_n=\infty\}$, which completes the argument.
\end{proof}

\subsection{Proof of Theorem~\ref{T:conv}}

We start by proving that {\ref{T:conv:a}}  yields that $\Delta X$ is extended locally integrable on $D$. By localization, in conjunction with Lemma~\ref{L:ELI}\ref{L:ELI:2}, we may assume that $(\Delta X)^-\wedge X^-\le\Theta$ for some integrable random variable $\Theta$ and that $\sup_{t<\tau} |X_t| < \infty$. With $\rho_n=\inf\{t \geq 0:  X_t\le-n\}$ we have $X^{\rho_n}\ge -n-(\Delta X_{\rho_n})^-\1{\rho_n<\tau}$ and $X^{\rho_n}\ge -n-X_{\rho_n}^-\1{\rho_n<\tau}$. Hence $X^{\rho_n}  \geq -n-\Theta$ and thus, by Lemma~\ref{L:SMC}, $X^{\rho_n}$ is extended locally integrable and Lemma~\ref{L:ELI}\ref{L:ELI:3} yields that $\Delta X^{\rho_n}$ is as well, for each $n \in \N$. We have $\bigcup_{n\in\N}\{\rho_n=\tau\}=\Omega$, and another application of Lemma~\ref{L:ELI}\ref{L:ELI:2} yields the implication.

We now verify the claimed implications.

{\ref{T:conv:a}} $\Longrightarrow$ {\ref{T:conv:a'}}:  Obvious.

{\ref{T:conv:a'}} $\Longrightarrow$ {\ref{T:conv:a}}:   By localization we may assume that $(\Delta X)^-\wedge X^-\le\Theta$ for some integrable random variable $\Theta$ and that $\sup_{t<\tau} X_t^- < \infty$ on $\Omega$.
 With $\rho_n=\inf\{t \geq 0:  X_t\le-n\}$ we have $X^{\rho_n}  \geq -n-\Theta$, for each $n \in \N$.
 The supermartingale convergence theorem (Lemma~\ref{L:SMC}) now implies that $X$ converges.

{\ref{T:conv:a}} $\Longrightarrow$ {\ref{T:conv:b'}}: This is an application of Lemma~\ref{L:ELI}\ref{L:ELI:3}, after recalling that {\ref{T:conv:a}} implies that $\Delta X$ is extended locally integrable on $D$.

{\ref{T:conv:b'}} $\Longrightarrow$ {\ref{T:conv:a}}: This is an application of a localization argument and the supermartingale convergence theorem stated in Lemma~\ref{L:SMC}.

{\ref{T:conv:a}} $\Longrightarrow$ {\ref{T:conv:b''}} \& {\ref{T:conv:c}} \& {\ref{T:conv:f}}:
By localization, we may assume that $|\Delta X|\le \Theta$ for some integrable random variable $\Theta$ and that $X = X^{\rho}$ with $\rho=\inf\{t \geq 0:|X_t|\ge \kappa\}$ for some fixed $\kappa \geq 0$. Next, observe that $X\ge -\kappa-\Theta$. Lemma~\ref{L:SMC} yields that $G=\lim_{t\to \tau}X_t$ exists in $\R$ and that the process $\overline X=X\oo_{\lc0,\tau\lc}+G\oo_{\lc\tau,\infty\lc}$, extended to $[0,\infty]$ by $\overline X_\infty=G$, is a supermartingale on $[0,\infty]$. Let $\overline X=\overline M-\overline A$ denote its canonical decomposition. Then $A_{\tau-}=\overline A_\infty<\infty$ and $[X,X]_{\tau-}= [\overline X,\overline X]_\infty < \infty$. Moreover, since $\overline X$ is a special semimartingale on $[0,\infty]$, Proposition~II.2.29 in~\cite{JacodS} yields $(x^2\wedge|x|)*\nu^X_{\tau-}=(x^2\wedge|x|)*\nu^{\overline X}_\infty<\infty$. Thus~{\ref{T:conv:c}} and~{\ref{T:conv:f}} hold. Now, {\ref{T:conv:b''}} follows again by an application of Lemma~\ref{L:ELI}\ref{L:ELI:3}.

{\ref{T:conv:b''}} $\Longrightarrow$ {\ref{T:conv:a}}:  By Lemma~\ref{L:ELI}\ref{L:ELI:4} we may assume that $A = 0$, so that $-X$ is a local supermartingale. The result then follows again from Lemma~\ref{L:SMC} and  Lemma~\ref{L:ELI}\ref{L:ELI:3}.

{\ref{T:conv:c}} $\Longrightarrow$ {\ref{T:conv:a}}:
The process $B=[X^c,X^c]+(x^2\wedge|x|)*\nu^X+A$ is predictable and converges on~$D$. Hence, by Lemma~\ref{L:ELI}\ref{L:ELI:4}, $B$ is  extended locally bounded on $D$. By localization we may thus assume that $B\le\kappa$ for some constant $\kappa>0$. Lemma~\ref{L:CC} implies that $X^c$ converges and Lemma~\ref{L:BJ} implies that $x\oo_{|x|\le 1}*(\mu^X-\nu^X)$ converges. Furthermore,
\[
\E\left[\, |x|\oo_{|x|>1}*\mu^X_{\tau-}\, \right] = \E\left[\, |x|\oo_{|x|>1}*\nu^X_{\tau-}\, \right] \le \kappa,
\]
whence $|x|\oo_{|x|>1}*(\mu^X-\nu^X)=|x|\oo_{|x|>1}*\mu^X-|x|\oo_{|x|>1}*\nu^X$ converges. We deduce that~$X$ converges. It now suffices to show that $\Delta X$ is extended locally integrable. Since
\begin{align*}
\sup_{t<\tau} |\Delta X_t|
\leq 1 + |x|\oo_{|x|\ge 1}*\mu^X_{\tau-},
\end{align*}
we have $\E[\sup_{t<\tau} |\Delta X_t|] \le 1+\kappa$. 

{\ref{T:conv:f}} $\Longrightarrow$ {\ref{T:conv:a}}:  By a localization argument we may assume that $(\Delta X)^- \wedge X^-\leq \Theta$ for some integrable random variable~$\Theta$.  Moreover, since $[X,X]_{\tau-}<\infty$ on $D$, $X$ can only have finitely many large jumps on $D$. Thus after further localization we may assume that $X = X^{\rho}$, where $\rho=\inf\{t \geq 0:|\Delta X_t|\ge \kappa_1\}$ for some large $\kappa_1>0$. Now, Lemmas~\ref{L:CC} and~\ref{L:BJ} imply that  $X' = X^c + x \oo_{|x| \leq \kappa_1} * (\mu^X-\nu^X)$  converges on $D$.  Hence Lemma~\ref{L:ELI}\ref{L:ELI:3} and a further localization argument let us assume that $|X'|\le\kappa_2$ for some constant $\kappa_2>0$.  Define $\widehat X = x \oo_{x<-\kappa_1} * (\mu^X-\nu^X)$ and suppose for the moment we know that~$\widehat X$ converges on~$D$. Consider the decomposition
\begin{equation} \label{eq:T:conv:ftob}
X = X' + \widehat X + x\oo_{x>\kappa_1}*\mu^X-x\oo_{x>\kappa_1}*\nu^X - A.
\end{equation}
The first two terms on the right-hand side converge on $D$, as does the third term since $X=X^\rho$. However, since $\limsup_{t\to\tau}X_t>-\infty$ on $D$ by hypothesis, this forces also the last two terms to converge on~$D$, and we deduce~\ref{T:conv:a} as desired.  It remains to prove that $\widehat X$ converges on~$D$, and for this we will rely repeatedly on the equality $X=X^\rho$ without explicit mentioning. In view of~\eqref{eq:T:conv:ftob} and the bound $|X'|\le\kappa_2$, we have
\[
\widehat X  \geq X - \kappa_2  -  x \oo_{x>\kappa_1} * \mu^X = X - \kappa_2 -(\Delta X_\rho)^+\oo_{\lc\rho,\tau\lc}.
\]
Moreover,  by definition of $\widehat X$  and $\rho$ we have $\widehat X\ge 0$ on $\lc0,\rho\lc$; hence $\widehat X\ge \Delta X_\rho\oo_{\lc\rho,\tau\lc}$. We deduce that on $\{\rho<\tau\text{ and }\Delta X_\rho<0\}$ we have $\widehat X^-\le X^-+\kappa_2$ and $\widehat X^-\le(\Delta X)^-$. On the complement of this event, it follows directly from the definition of $\widehat X$ that $\widehat X\ge0$. To summarize, we have $\widehat X^- \leq   (\Delta X)^-  \wedge  X^-  +  \kappa_2 \leq \Theta + \kappa_2.$  Lemma~\ref{L:SMC} now implies that $\widehat X$ converges, which proves the stated implication.

{\ref{T:conv:a}} \& {\ref{T:conv:f}} $\Longrightarrow$    {\ref{T:conv:g}}: We now additionally assume that $X$ is constant on $\lc\tau_J,\tau\lc$. First, note that $\Ecal(X)$ changes sign finitely many times on $D$ since $ \oo_{x < -1} * \mu_{\tau-}^X \leq  x^2  *\mu_{\tau-}^X < \infty$ on $D$. Therefore, it is sufficient to check that $\lim_{t\to\tau}|\Ecal(X)_t|$ exists in $(0,\infty)$ on $D \cap \{\tau_J = \infty\}.$  However, this follows from the fact that  $\log|\Ecal(X)|=X - [X^c,X^c] /2 - (x - \log |1+x|)  * \mu^X$ on $\lc 0, \tau_J\lc$ and the inequality $x-\log(1+x)\le x^2$ for all $x\ge-1/2$.

{\ref{T:conv:g}}  $\Longrightarrow$    {\ref{T:conv:a'}}: Note that we have  $\lim_{t \to \tau} X_t - [X^c,X^c]_t/2 - (x - \log (1+x))\oo_{x>-1} * \mu_t^X$ exists in $\R$ on $D$, which then yields the implication.
\qed

\subsection{Proof of Theorem~\ref{T:convYX}}

The proof relies on a number of intermediate lemmas. We start with a special case of Markov's inequality that is useful for estimating conditional probabilities in terms of unconditional probabilities. This inequality is then applied in a general setting to control conditional probabilities of excursions of convergent processes.

\begin{lemma}[A Markov type inequality] \label{L:cprob}
Let $\Gcal\subset\Fcal$ be a sub-sigma-field, and let $G\in\Gcal$, $F\in\Fcal$, and $\delta>0$. Then
\[
\P\left(\oo_G\,\P( F\mid\Gcal ) \geq \delta \right)\ \le\ \frac{1}{\delta}\P(G \cap F).
\]
\end{lemma}

\begin{proof} 
We have $\P(G\cap F) = \E\left[ \oo_G\,\P(F\mid\Gcal) \right] \ge \delta\, \P\left(\oo_G\,\P(F\mid\Gcal)\geq \delta\right)$.
\end{proof}

\begin{lemma} \label{L:cprob2}
Let $\tau$ be a foretellable time, let $W$ be a measurable process on $\lc0,\tau\lc$, and let $(\rho_n)_{n\in\N}$ be a nondecreasing sequence of stopping times with $\lim_{n\to\infty}\rho_n\ge\tau$. Suppose the event
\[
C=\Big\{\lim_{t\to\tau} W_t=0 \text{ and } \rho_n<\tau \text{ for all }n\in\N\Big\}
\]
lies in $\Fcal_{\tau-}$. Then for each $\varepsilon>0$,
\[
\P\left( W_{\rho_n} \le \varepsilon \mid \Fcal_{\rho_n-}\right) \ge \frac{1}{2} \quad\text{for infinitely many }n\in\N
\]
holds almost surely on $C$.
\end{lemma}

\begin{proof}
By Theorem~IV.71 in~\cite{Dellacherie/Meyer:1978}, $\tau$ is almost surely equal to some predictable time $\tau'$. We may thus assume without loss of generality that $\tau$ is already predictable. Define events $F_n = \{W_{\rho_n} > \varepsilon \text{ and } \rho_n<\tau\}$ and $G_n = \{\P(C\mid\Fcal_{\rho_n-}) > 1/2\}$ for each $n \in \N$ and some fixed $\varepsilon > 0$.  By Lemma~\ref{L:cprob}, we have
\begin{equation}\label{eq:L:cprob2}
\P\left( \oo_{G_n} \P( F_n \mid \Fcal_{\rho_n-}) > \frac{1}{2} \right) \le 2\,\P(G_n\cap F_n) \le 2\,\P(F_n\cap C) + 2\,\P(G_n\cap C^c).
\end{equation}
Clearly, we have $\lim_{n\to \infty}\P(F_n\cap C)= 0$. Also, since $\rho_\infty=\lim_{n\to\infty}\rho_n\ge\tau$, we have $\lim_{n\to\infty}\P(C\mid\Fcal_{\rho_n-})=\P(C\mid\Fcal_{\rho_\infty-})=\oo_C$. Thus $\oo_{G_n}=\oo_C$ for all sufficiently large $n \in \N$, and hence $\lim_{n\to\infty}\P(G_n\cap C^c)=0$ by bounded convergence. The left-hand side of~\eqref{eq:L:cprob2} thus tends to zero as $n$ tends to infinity, so that, passing to a subsequence if necessary, the Borel-Cantelli lemma yields $\oo_{G_n} \P( F_n \mid \Fcal_{\rho_n-})\le1/2$ for all but finitely many $n \in \N$. Thus, since $\oo_{G_n}=\oo_C$ eventually, we have $\P( F_n \mid \Fcal_{\rho_n-})\le1/2$ for infinitely many $n \in \N$ on $C$. Since $\tau$ is predictable we have $\{\rho_n<\tau\}\in\Fcal_{\rho_n-}$ by Theorem~IV.73(b) in \cite{Dellacherie/Meyer:1978}. Thus $\P( F_n \mid \Fcal_{\rho_n-})=\P( W_{\rho_n}>\varepsilon \mid \Fcal_{\rho_n-})$ on $C$, which yields the desired conclusion.
\end{proof}

Returning to the setting of Theorem~\ref{T:convYX}, we now show that $\gamma$ vanishes asymptotically on the event~\eqref{T:convYX:4}.

\begin{lemma} \label{L:convYX:1}
Under Assumption~\ref{A:2}, we have $\lim_{t\to\tau}\gamma_t=0$ on~\eqref{T:convYX:4}.
\end{lemma}

\begin{proof}
As in the proof of Lemma~\ref{L:cprob} we may assume that $\tau$ is predictable. We now argue by contradiction. To this end, assume there exists $\varepsilon>0$ such that $\P(C)>0$ where $C=\{\gamma_t\ge 2 \varepsilon\text{ for infinitely many $t$}\} \cap \eqref{T:convYX:4}$. Let $(\rho_n)_{n\in\N}$ be a sequence of predictable times covering the predictable set $\{\gamma\ge2\varepsilon\}$. By \eqref{eq:DY} and since $X$ and $Y$ are c\`adl\`ag, any compact subset of $[0,\tau)$ can only contain finitely many time points~$t$ for which $\gamma_t\ge 2\varepsilon$. We may thus take the $\rho_n$ to satisfy $\rho_n<\rho_{n+1}<\tau$ on $C$ for all $n\in\N$, as well as $\lim_{n\to\infty}\rho_n\ge\tau$.

We now have, for each $n\in\N$ on $\{\rho_n<\tau\}$,
\begin{align*}
0
&= \int x\nu^X(\{\rho_n\}, \dd x) 
\le -(1-e^{-\varepsilon}) \P\left(\Delta X_{\rho_n} \le e^{-\varepsilon}-1\mid\Fcal_{\rho_n-}\right) + \int x\oo_{x>0}\, \nu^X(\{\rho_n\},\dd x) \\
&\le -(1-e^{-\varepsilon}) \P\left(\Delta Y_{\rho_n} \le \varepsilon\mid\Fcal_{\rho_n-}\right) + \int x\oo_{x>0}\, \nu^X(\{\rho_n\},\dd x),
\end{align*}
where the equality uses the local martingale property of $X$, the first inequality is an elementary bound involving Equation~II.1.26 in~\citet{JacodS}, and the second inequality follows from~\eqref{eq:DY}. 

Thus on $C$,
\begin{equation} \label{eq:convYX:1:1}
x \oo_{x \geq 0 \vee (e^{\varepsilon-\gamma}-1)}*\nu^X_{\tau-}
\ge \sum_{n\in\N}\int x\oo_{x>0}\, \nu^X(\{\rho_n\},\dd x) 
\ge (1-e^{-\varepsilon})\sum_{n\in\N}\P\left(\Delta Y_{\rho_n} \le \varepsilon\mid\Fcal_{\rho_n-}\right).
\end{equation}
With $W = \Delta Y$, Lemma~\ref{L:cprob2} implies that the right-hand side of~\eqref{eq:convYX:1:1} is infinite almost surely on $C$. 

We now argue that the left-hand is finite almost surely on \eqref{T:convYX:4} $\supset C$, yielding the contradiction. To this end, since $\lim_{t\to\tau}\Delta Y_t=0$ on~\eqref{T:convYX:4}, we have $\oo_{x > e^{\varepsilon - \gamma} - 1} * \mu^X_{\tau-} < \infty$ on~\eqref{T:convYX:4}. Lemma~\ref{L:BJ} and an appropriate localization argument applied to the random measure
\[
\mu=\oo_{0 \vee (e^{\varepsilon-\gamma}-1)\le x\le \kappa} \oo_{\lc 0, \tau\lc}\,\mu^X
\]
yield $x \oo_{0 \vee (e^{\varepsilon-\gamma}-1)\le x\le \kappa}*\nu^X_{\tau-}<\infty$; here $\kappa$ is as in Theorem~\ref{T:convYX}.  Since also $x\oo_{x>\kappa}*\nu^X_{\tau-}<\infty$ on~\eqref{T:convYX:4} by definition, the left-hand side of~\eqref{eq:convYX:1:1} is finite. 
\end{proof}

\begin{lemma} \label{L:convYX:0 new}
Fix $\varepsilon \in (0,1)$. Under Assumption~\ref{A:2}, we have
\[
[X^c, X^c]_{\tau-} + (\log(1+x)+\gamma)^2 \oo_{|x| \leq \varepsilon} * \nu^X_{\tau-}  - \log(1+x) \oo_{x \leq -\varepsilon} * \nu^X_{\tau-} + x  \oo_{x \geq  \varepsilon} * \nu^X_{\tau-}  < \infty
\]
 on~\eqref{T:convYX:4}.
 \end{lemma}

\begin{proof}
As in Lemma~\ref{L:convYX} we argue that the condition \ref{T:conv:a} in Theorem~\ref{T:conv} holds with $X$ replaced by $-Y$. Using the equivalence with Theorem~\ref{T:conv}\ref{T:conv:c}, we obtain that
 $[X^c, X^c]_{\tau-}  = [Y^c, Y^c]_{\tau-} < \infty$ and 
\begin{align}\label{eq: L:convYX:0 new proof}
\Big( (\log(1+x)+\gamma)^2 \wedge |\log(1+x)+\gamma|\Big) * \nu^X_{\tau-}  + \sum_{s<\tau} (\gamma_s^2\wedge\gamma_s)\1{\Delta X_s=0}
= (y^2\wedge|y|)*\nu^Y_{\tau-}  < \infty
\end{align}
on \eqref{T:convYX:4}, where the equality in~\eqref{eq: L:convYX:0 new proof} follows from \eqref{eq:DY}.  Now, by localization, Lemma~\ref{L:convYX:1}, and Lemma~\ref{L:ELI}\ref{L:ELI:4}, we may assume that $\sup_{t < \tau} \gamma_t$ is bounded.  We then obtain from~\eqref{eq: L:convYX:0 new proof} that $(\log(1+x)+\gamma)^2 \oo_{|x| \leq \varepsilon} * \nu^X_{\tau-} < \infty$ on \eqref{T:convYX:4}. 

Next, note that
\begin{align*}  
  -\log(1+x)\oo_{x\le -\varepsilon} * \nu^X_{\tau-} &=   -\log(1+x)\oo_{x\le-\varepsilon}  \oo_{\{\gamma<-\log(1-\varepsilon)/2\}}  * \nu^X_{\tau-} \\
  	&\qquad +  \sum_{t < \tau} \int -\log(1+x)\oo_{x\le-\varepsilon}  \oo_{\{\gamma\geq -\log(1-\varepsilon)/2\}}\,  \nu^X(\{t\},\dd x) < \infty
\end{align*}
on \eqref{T:convYX:4}. Indeed, an argument based on \eqref{eq: L:convYX:0 new proof} shows that the first summand is finite. The second summand is also finite since it consists of finitely many terms due to Lemma~\ref{L:convYX:1}, each of which is finite. The latter follows since $(x-\log(1+x))*\nu^X$ is a finite-valued process by assumption and $\int|x|\nu^X(\{t\},dx)<\infty$ for all $t<\tau$ due to the local martingale property of $X$.
Finally, a calculation based on \eqref{eq: L:convYX:0 new proof} yields $x\oo_{\varepsilon\le x\le \kappa}*\nu^X_{\tau-}<\infty$ on the event \eqref{T:convYX:4}, where $\kappa$ is as in Theorem~\ref{T:convYX}. This, together with the definition of \eqref{T:convYX:4}, implies that $x  \oo_{x \geq  \varepsilon} * \nu^X_{\tau-}  < \infty$ there, completing the proof.
\end{proof}

We are now ready to verify the claimed inclusions of Theorem~\ref{T:convYX}.

\eqref{T:convYX:1} $\subset$ \eqref{T:convYX:2}: The implication \ref{T:conv:a} $\Longrightarrow$ \ref{T:conv:g} of Theorem~\ref{T:conv} shows that $\Ecal(X)_{\tau-}>0$ on~\eqref{T:convYX:1}. The desired inclusion now follows from~\eqref{eq:VV}.

\eqref{T:convYX:2} $\subset$ \eqref{T:convYX:1}: By the inclusion \eqref{T:conv2:2} $\subset$ \eqref{T:conv2:1} of Corollary~\ref{C:conv2} and the implication \ref{T:conv:a} $\Longrightarrow$ \ref{T:conv:g} of Theorem~\ref{T:conv}, $X$ converges and $\Ecal(X)_{\tau-}>0$ on \eqref{T:convYX:2}. Hence by \eqref{eq:VV}, $Y$ also converges on \eqref{T:convYX:2}.

\eqref{T:convYX:1}  $\cap$ \eqref{T:convYX:2} $\subset$ \eqref{T:convYX:3}: Obvious.

\eqref{T:convYX:3} $\subset$ \eqref{T:convYX:2}: The inclusion \eqref{T:conv2:1} $\subset$ \eqref{T:conv2:2} of Corollary~\ref{C:conv2} implies $[X^c,X^c]_{\tau-}+(x^2\wedge|x|)*\nu^X_{\tau-}<\infty$ on~\eqref{T:convYX:3}. Since also $-\log(1+x)\oo_{x\le -\eta}*\nu^X_{\tau-} < \infty$ on~\eqref{T:convYX:3} by definition, the desired inclusion follows.

\eqref{T:convYX:1} $\cap$ \eqref{T:convYX:2} $\subset$ \eqref{T:convYX:4}: Obvious.

\eqref{T:convYX:4} $\subset$ \eqref{T:convYX:1}: We need to show that $X$ converges on \eqref{T:convYX:4}. By Theorem~\ref{T:conv} it is sufficient to argue that $[X^c,X^c]_{\tau-}+(x^2\wedge|x|)*\nu^X_{\tau-}<\infty$ on \eqref{T:convYX:4}.  Lemma~\ref{L:convYX:0 new} yields directly that $[X^c,X^c]_{\tau-}< \infty$, so we focus on the jump component. To this end, using that $\int x\nu^X(\{t\},dx)=0$ for all $t<\tau$, we first observe that, for fixed $\varepsilon\in (0,1)$,
\begin{align*}
\gamma_t
&= \int (x-\log(1+x)) \,\nu^X(\{t\},\dd x) \\
&\le \frac{1}{1-\varepsilon} \int x^2\oo_{|x|\le \varepsilon}\,\nu^X(\{t\},\dd x) + \int x\oo_{x>\varepsilon}\,\nu^X(\{t\},\dd x)  +  \int -\log(1+x)\oo_{x<-\varepsilon}\,\nu^X(\{t\},\dd x)
\end{align*}
for all $t<\tau$.
Letting $\Theta_t$ denote the last two terms for each $t < \tau$, Lemma~\ref{L:convYX:0 new} implies that $\sum_{t<\tau}\Theta_t<\infty$, and hence also $\sum_{t<\tau}\Theta_t^2<\infty$, hold on~\eqref{T:convYX:4}. Furthermore, the inequality $(a+b)^2\le 2a^2 + 2b^2$ yields that
\begin{align} \label{eq: sum gamma2}
\sum_{t<\tau_n}  \gamma_t^2
&\le  \frac{2}{(1-\varepsilon)^2} \sum_{t<\tau_n}  \left( \int x^2\oo_{ |x| \le \varepsilon}\,\nu^X(\{t\},dx) \right)^2 + 2 \sum_{t<\tau_n}  \Theta_t^2 
 \le  \frac{2 \varepsilon^2}{(1-\varepsilon)^2}  x^2\oo_{|x|\le \varepsilon} * \nu^X_{\tau_n} + 2 \sum_{t<\tau}  \Theta_t^2
\end{align} 
for all $n \in \N$, where $(\tau_n)_{n\in\N}$ denotes an announcing sequence for~$\tau$.

Also observe that, for all $n \in \N$,
\begin{align*}
\frac{1}{16}  x^2 \oo_{|x|\le \varepsilon} * \nu^X_{\tau_n} 
&\leq (\log(1+x))^2 \oo_{|x|\leq \varepsilon} * \nu^X_{\tau_n} 
\leq 2 (\log(1+x) + \gamma)^2 \oo_{|x| \leq \varepsilon} * \nu^X_{\tau_n} + 2 \sum_{t \leq \tau_n} \gamma_t^2,
\end{align*}
which yields, thanks to \eqref{eq: sum gamma2},
\begin{align*}
	\left(\frac{1}{16} - \frac{4 \varepsilon^2}{(1-\varepsilon)^2} \right) x^2 \oo_{|x|\le \varepsilon} * \nu^X_{\tau_n} \leq 2 (\log(1+x) + \gamma)^2 \oo_{|x| \le \varepsilon} * \nu^X_{\tau_n} + 4 \sum_{t<\tau}  \Theta_t^2.
\end{align*}
Choosing $\varepsilon$ small enough and letting $n$ tend to infinity, we obtain that $x^2 \oo_{|x|\le \varepsilon} * \nu^X_{\tau} < \infty$ on \eqref{T:convYX:4} thanks to Lemma~\ref{L:convYX:0 new}. The same lemma also yields  $|x| \oo_{|x|\geq \varepsilon} * \nu^X_{\tau} < \infty$, which concludes the proof.
\qed

\section{Novikov-Kazamaki conditions} \label{S:NK}

We now apply our convergence results to prove general Novikov-Kazamaki type conditions. Throughout this section we fix a nonnegative local martingale $Z$ with $Z_0=1$, and define $\tau_0=\inf\{t \geq 0:Z_t=0\}$. We assume that $Z$ does not jump to zero, meaning that $Z_{\tau_0-}=0$ on $\{\tau_0<\infty\}$. The stochastic logarithm $M=\Lcal(Z)$ is then a local martingale on $\lc0,\tau_0\lc$ with $\Delta M>-1$; see Appendix~\ref{A:SE}. We let $\tau_\infty=\lim_{n\to \infty}\inf\{t \geq 0:Z_t\ge n\}$ denote the explosion time of $Z$; clearly, $\P(\tau_\infty<\infty)=0$. To distinguish between different probability measures we now write $\E_R[\,\cdot\,]$ for the expectation operator under a probability measure $R$.

\subsection{General methodology} \label{S:method}

The idea of our approach is to use $Z$ as the density process of a measure change, without knowing {\em a priori} whether~$Z$ is a uniformly integrable martingale. This can be done whenever the filtration $\F$ is the right-continuous modification of a standard system, for instance if $\F$ is the right-continuous canonical filtration on the space of (possibly explosive) paths; see \citet{F1972} and \citet{Perkowski_Ruf_2014}.  This assumption rules out that $\F$ is augmented with the $\P$--nullsets, which is one reason for avoiding the ``usual conditions'' in the preceding theory. We assume $\F$ has this property, and emphasize that no generality is lost; see Appendix~\ref{app:embed}. The resulting measure, denoted $\Q$ throughout this section, is sometimes called the {\em F\"ollmer measure}; Appendix~\ref{S:follmer} reviews some relevant facts. Its crucial property is that $Z$ explodes with positive $\Q$--probability if and only if $Z$ is not a uniformly integrable martingale under~$\P$. This is where our convergence results enter the picture: under Novikov-Kazamaki type conditions, they are used to exclude explosions under~$\Q$. 

The following ``code book'' contains the basic definitions and facts that are used to translate previous sections into the current setting. It will be used extensively throughout this section.
\begin{enumerate}
\item Consider the process
\begin{equation} \label{eq:N}
N = -M+\langle M^c,M^c\rangle+\frac{x^2}{1+x} * \mu^M
\end{equation}
on $\lc0,\tau_0\lc$. Theorem~\ref{T:conv} in conjunction with Theorems~\ref{T:SE}, \ref{T:reciprocal}, and \ref{T numeraire} readily imply:
\begin{itemize}
\item $1/Z=\Ecal(N)$ and $N$ is a local martingale on $\lc0,\tau_\infty\lc$ under~$\Q$.
\item $Z$ is a uniformly integrable martingale under~$\P$ if and only if $\Q(\lim_{t\to\tau_\infty}N_t\text{ exists in }\R)=1$.
\end{itemize}
\item Consider the process
\begin{equation} \label{eq:L}
L = -M+\langle M^c,M^c\rangle+(x-\log(1+x)) * \mu^M + ((1+x)\log(1+x)-x)*\nu^M
\end{equation}
on $\lc0,\tau_0\lc$. This is always well-defined, but the last term may be infinite-valued. By Theorem~\ref{T:reciprocal} the last term equals $(y-\log(1+y))*\widehat\nu^N$, where $\widehat\nu^N=\nu^N/(1+y)$. If it is finite-valued, we have:
\begin{itemize}
\item $\widehat\nu^N$ is the predictable compensator of~$\mu^N$ under~$\Q$; see Lemma~\ref{L:predc}.
\item $L=N^c+\log(1+y)*(\mu^N-\widehat\nu^N)$. Note that $\log(1+y)$ lies in $G_{\rm loc}(\mu^N)$ under~$\Q$ since both $y$ and $y-\log(1+y)$ do.
\item We are in the setting of Assumption~\ref{A:2} under~$\Q$, with $\tau=\tau_\infty$, $X=N$ (hence $\nu^X=\widehat\nu^N$), and $Y=L$.
\end{itemize}
\end{enumerate}
With this ``code book'' at our disposal we may now give a quick proof of the following classical result due to \citet{Lepingle_Memin_Sur}. The proof serves as an illustration of the general technique and as a template for proving the more sophisticated results presented later on.

\begin{theorem}[The \citet{Lepingle_Memin_Sur} conditions] \label{T:LepMem}
On $\lc0,\tau_0\lc$, define the nondecreasing processes
\begin{align}
A &= \frac{1}{2}\langle M^c,M^c\rangle + \Big(\log(1+x) - \frac{x}{1+x}\Big)*\mu^M; \label{eq:LM_A} \\
B &= \frac{1}{2}\langle M^c,M^c\rangle + ((1+x)\log(1+x) - x)*\nu^M.  \label{eq:LM_B} 
\end{align}
If either $\E_\P[e^{A_{\tau_0-}}]<\infty$ or $\E_\P[e^{B_{\tau_0-}}]<\infty$, then $Z$ is a uniformly integrable martingale.
\end{theorem}

\begin{proof}
We start with the criterion using $A$. A brief calculation gives the identity
\[
A = \left( M - \frac{1}{2}\langle M^c,M^c\rangle - (x-\log(1+x))*\mu^M \right) + \left( -M + \langle M^c,M^c\rangle + \frac{x^2}{1+x}*\mu^M \right)
=\log Z + N.
\]
Thus, using also that $A$ is nondecreasing, we obtain
\[
\infty > \sup_{\sigma\in\Tcal} \E_\P\left[e^{A_\sigma}\1{\sigma<\tau_0}\right]
= \sup_{\sigma\in\Tcal} \E_\Q\left[e^{N_\sigma}\1{\sigma<\tau_\infty}\right].
\]
The implication \ref{T:conv1:c} $\Longrightarrow$ \ref{T:conv1:a} in Corollary~\ref{C:conv001} now shows that $\lim_{t\to\tau_\infty}N_t$ exists in~$\R$, $\Q$--almost surely. Thus $Z$ is a uniformly integrable martingale under~$\P$.

The criterion using $B$ is proved similarly. First, the assumption implies that~$B$ and hence~$L$ are finite-valued. Theorem~\ref{T:reciprocal} and a calculation yield
\[
B = \frac{1}{2}\langle N^c,N^c\rangle + ((1+\phi(y))\log(1+\phi(y)) - \phi(y))*\nu^N
=\frac{1}{2}\langle N^c,N^c\rangle + (y-\log(1+y))*\widehat\nu^N,
\]
where $\phi$ is the involution in~\eqref{eq:phi}. Observing that $L = N - (y - \log(1+y))* (\mu^N-\widehat\nu^N)$, we obtain $-\log Z = \log \Ecal(N) = L - B$, and hence
\[
\infty > \sup_{\sigma\in\Tcal} \E_\P\left[e^{B_\sigma}\1{\sigma<\tau_0}\right]
= \sup_{\sigma\in\Tcal} \E_\Q\left[e^{L_\sigma}\1{\sigma<\tau_\infty}\right].
\]
The implication \ref{T:convYX2:c} $\Longrightarrow$ \ref{T:convYX2:a} in Corollary~\ref{C:convYX2} now shows that $N$ converges $\Q$--almost surely.
\end{proof}

\begin{remark}We make the following observations concerning Theorem~\ref{T:LepMem}:
\begin{itemize}  
\item Note that we have the following representation by Theorem~\ref{T:reciprocal}, with $\widehat{\nu}^N = \nu^N/(1+y)$.
\begin{align*}
A &= \frac{1}{2}\langle N^c,N^c\rangle + (y- \log(1+y))*\mu^N;\\
B &= \frac{1}{2}\langle N^c,N^c\rangle + (y-\log(1+y) )*\widehat\nu^N.
\end{align*}
Thus, by Lemma~\ref{L:predc}, $B$ is the predictable compensator (if it exists) of $A$ under the F\"ollmer measure $\Q$; see also Remarque~III.12 in \citet{Lepingle_Memin_Sur}.
\item Any of the two conditions in Theorem~\ref{T:LepMem} implies that   $Z_\infty>0$ and thus $\tau_0 = \infty$, thanks to Theorem~\ref{T:conv}.  This has already been observed in Lemmes~III.2 and III.8
 in \citet{Lepingle_Memin_Sur}.
\item For the condition involving $B$, \citet{Lepingle_Memin_Sur} allow $Z$ to jump to zero. This case can be treated using our approach as well, albeit with more intricate arguments. For reasons of space, we focus on the case where $Z$ does not jump to zero. \qed
\end{itemize} 
\end{remark}

\citet{Protter_Shimbo} and \citet{Sokol2013_optimal} observe that if $\Delta M \geq -1+\delta$ for some $\delta \geq 0$ then the expressions $\log(1+x) - x/(1+x)$ and $(1+x) \log(1+x) - x$ appearing in \eqref{eq:LM_A} and \eqref{eq:LM_B}, respectively, can be bounded by simplified (and more restrictive) expressions.

\subsection{An abstract characterization and its consequences}

In a related paper, \citet{Lepingle_Memin_integrabilite} embed the processes $A$ and $B$ from Theorem~\ref{T:LepMem} into a parameterized family of processes $A^a$ and $B^a$, which can be defined on $\lc0,\tau_0\lc$ for each $a\in\R$ by
\begin{align}\label{eq:A}
A^{a} &= aM + \left( \frac{1}{2}-a\right) [M^c,M^c] + \left( \log(1+x) - \frac{ax^2+x}{1+x}\right) * \mu^M; \\
\nonumber
B^{a}  &= aM + \left( \frac{1}{2}-a\right) [M^c,M^c] - a\left( x-\log(1+x) \right) * \mu^M   +(1-a)\left((1+x)\log(1+x)  - x\right) * \nu^M.  \nonumber
\end{align}
Note that $A^0=A$ and $B^0=B$. They then prove that uniform integrability of $(e^{A^a_\sigma})_{\sigma\in\Tcal}$ or $(e^{B^a_\sigma})_{\sigma\in\Tcal}$ for some $a\in[0,1)$ implies that $Z$ is a uniformly integrable martingale. Our present approach sheds new light on this result and enables us to strengthen it. A key observation is that $A^a$ and $B^a$ satisfy the following identities, which extend those for $A$ and $B$ appearing in the proof of Theorem~\ref{T:LepMem}. Recall that $N$ and $L$ are given by~\eqref{eq:N} and~\eqref{eq:L}.

\begin{lemma} \label{L:ABdecomp}
The processes $A^a$ and $B^a$ satisfy, $\Q$--almost surely on $\lc0,\tau_\infty\lc$,
\begin{align*}
A^a &= \log Z + (1-a)N;\\
B^{a} &= \log Z + (1-a)L.
\end{align*}
\end{lemma}

\begin{proof}
The identities follow from Theorem~\ref{T:reciprocal} and basic computations that we omit here.
\end{proof}

We now state a general result giving necessary and sufficient conditions for $Z$ to be a uniformly integrable martingale. It is a direct consequence of the convergence results in Section~\ref{S:convergence}. In combination with Lemma~\ref{L:ABdecomp} this will yield improvements of the result by \citet{Lepingle_Memin_integrabilite} and give insight into how far such results can be generalized.

\begin{theorem}[Abstract \citet{Lepingle_Memin_integrabilite} conditions]   \label{T:NK}
Let $f:\R\to\R_+$ be any nondecreasing function with $f(x)\ge x$ for all sufficiently large~$x$, and let $\epsilon \in \{-1,1\}$, $\kappa > 0$, and $\eta \in (0,1)$.
Then the following  conditions are equivalent:
\begin{enumerate}[label={\rm(\alph*)},ref={\rm(\alph*)}]
	\item\label{T:NK:1} $Z$ is a uniformly integrable martingale.
	\item\label{T:NK:2} There exists an optional process $U$, extended locally integrable on $\lc0,\tau_\infty\lc$ under $\Q$, such that 
\begin{equation}  \label{eq:4.7}
\sup_{\sigma\in\Tcal} \E_\P\left[ Z_\sigma f(\epsilon N_\sigma-U_\sigma)\1{\sigma < \tau_0}\right] < \infty.
\end{equation}
\end{enumerate}
Moreover, the following conditions are equivalent:
\begin{enumerate}[label={\rm(\alph*)},ref={\rm(\alph*)},resume]
\item\label{T:NK:3}  $Z$ is a uniformly integrable martingale and  $(1+x)\log(1+x) \oo_{x>\kappa} * \nu^M$ is extended locally integrable on $\lc0,\tau_\infty\lc$ under $\Q$.
\item\label{T:NK:4} $(1+x)\log(1+x) \oo_{x > \kappa} * \nu^M$ is finite-valued,  $-x \oo_{x<-\eta} * \nu^M$ is extended locally integrable on $\lc0,\tau_\infty\lc$ under $\Q$, and
\begin{align*} 
	\sup_{\sigma \in \mathcal{T}} \E_\P\left[Z_\sigma f(\epsilon L_\sigma-U_\sigma)\right] < \infty
\end{align*}
for some optional process $U$, extended locally integrable on $\lc0,\tau_\infty\lc$ under~$\Q$.
\end{enumerate}
\end{theorem}

\begin{important remark}[Characterization of extended local integrability] \label{R:ELI Q P}
The extended local integrability under $\Q$ in Theorem~\ref{T:NK}\ref{T:NK:2}--\ref{T:NK:4} can also be phrased in terms of the ``model primitives'', that is, under $\P$. As this reformulation is somewhat subtle---in particular, extended local integrability under $\P$ is {\em not} equivalent---we have opted for the current formulation in terms of~$\Q$. A characterization under~$\P$ is provided in Appendix~\ref{App:Z}, which should be consulted by any reader who prefers to work exclusively under~$\P$. A crude, but simple, sufficient condition for $U$ to be extended locally integrable under~$\Q$ is that $U$ is bounded.
\qed
\end{important remark}

\begin{proof}[Proof of Theorem~\ref{T:NK}]
We use the ``code book'' from Subsection~\ref{S:method} freely. Throughout the proof, suppose $\epsilon=1$; the case $\epsilon=-1$ is similar. Observe that \eqref{eq:4.7} holds if and only if \eqref{eq:T:conv:exp} holds under~$\Q$ with $X=N$ and $\tau=\tau_\infty$. Since $\Delta N>-1$, Corollary~\ref{C:conv001} thus shows that \ref{T:NK:2} is equivalent to the convergence of $N$ under $\Q$, which is equivalent to~\ref{T:NK:1}.

To prove the equivalence of \ref{T:NK:3} and \ref{T:NK:4}, first note that $((1+x)\log(1+x) - x) * \nu^M$ is finite-valued under either condition, and hence so is~$L$. Note also the equalities
\[
(1+x)\log(1+x)\oo_{x\ge \kappa}*\nu^M=-\log(1+y)\oo_{y\le-\kappa/(1+\kappa)}*\widehat\nu^N
\]
and
\[
-x\oo_{x<-\eta}*\nu^M = y\oo_{y>\eta/(1-\eta)}*\widehat\nu^N.
\]
With the identifications in the ``code book'', \ref{T:NK:3} now states that the event~\eqref{T:convYX:3} has full probability under~$\Q$. By Theorem~\ref{T:convYX} this is equivalent to the event~\eqref{T:convYX:4} having full probability under $\Q$. Due to Corollary~\ref{C:conv001} this is equivalent to \ref{T:NK:4}.
\end{proof}

\begin{remark} \label{R:Uspecs}
We observe that, given condition~\ref{T:NK:1} or~\ref{T:NK:3} in Theorem~\ref{T:NK}, we may always choose $U$ to equal $(1-a)N$ or $(1-a)L$, respectively.
\qed
\end{remark}

\begin{corollary}[Generalized \citet{Lepingle_Memin_integrabilite} conditions] \label{C:NK1}
Fix $a\ne1$ and $\eta \in (0,1)$. The following condition is equivalent to Theorem~\ref{T:NK}\ref{T:NK:1}:
\begin{enumerate}[label={\rm(b$'\;\!$)},ref={\rm(b$'\;\!$)}]
\item\label{C:NK1:2} There exists an optional process $U$, extended locally integrable on $\lc0,\tau_\infty\lc$ under~$\Q$, such that 
\begin{equation}  \label{eq:NPN}
\sup_{\sigma\in\Tcal} \E_\P\left[ e^{A^a_\sigma - U_\sigma} \1{\sigma<\tau_0}\right] < \infty.
\end{equation}
\end{enumerate}
Moreover, the following conditions are equivalent to Theorem~\ref{T:NK}\ref{T:NK:3}:
\begin{enumerate}[label={\rm(d$'\;\!$)},ref={\rm(d$'\;\!$)}]
\item\label{C:NK1:4} $(1+x)\log(1+x) \oo_{x > \kappa} * \nu^M$ is finite-valued, $-x \oo_{x<-\eta} * \nu^M$ is extended locally integrable on $\lc0,\tau_\infty\lc$ under $\Q$, and
\begin{align}  \label{T:NK:eq1}
	\sup_{\sigma \in \mathcal{T}} \E_\P\left[e^{B^a_\sigma - U_\sigma}\1{\sigma<\tau_0}\right] < \infty
\end{align}
for some optional process $U$, extended locally integrable on $\lc0,\tau_\infty\lc$ under~$\Q$.
\end{enumerate}
\begin{enumerate}[label={\rm(d$''\;\!$)},ref={\rm(d$''\;\!$)}]
 \item\label{C:NK1:5} $(1+x)\log(1+x) \oo_{x > \kappa} * \nu^M$ is finite-valued and there exists an optional process $U$, extended locally integrable on~$\lc0,\tau_\infty\lc$ under~$\Q$, such that the family $(e^{B_\sigma^{a} - U_\sigma} \1{\sigma<\tau_0})_{\sigma\in\Tcal}$ is uniformly integrable.
 \end{enumerate}
If $a\leq 0$, these conditions are implied by the following:
 \begin{enumerate}[label={\rm(d$'''\;\!$)},ref={\rm(d$'''\;\!$)}]
\item\label{C:NK1:6}  \eqref{T:NK:eq1} holds for some optional process $U$ that is extended locally bounded on~$\lc0,\tau_\infty\lc$ under~$\Q$.
\end{enumerate}
\end{corollary}

\begin{proof}
In view of Lemma~\ref{L:ABdecomp}, the equivalences \ref{C:NK1:2} $\Longleftrightarrow$ Theorem~\ref{T:NK}\ref{T:NK:1} and \ref{C:NK1:4} $\Longleftrightarrow$ Theorem~\ref{T:NK}\ref{T:NK:3} follow by choosing $f(x)=e^{(1-a)x}$ for all $x \in \R$ and $\epsilon={\rm sign}(1-a)$ in Theorem~\ref{T:NK}.  The condition  \ref{C:NK1:5} is implied by 
Theorem~\ref{T:NK}\ref{T:NK:3}  thanks to the last statement in Corollary~\ref{C:convYX2}.  Now assume that  \ref{C:NK1:5} holds and assume for contradiction that $N$ does not converge under $\Q$. The assumed uniform integrability implies that for any $\varepsilon > 0$ there exists $\kappa_1 > 0$ such that
\begin{equation} \label{eq:C:NK1:001}
\sup_{\sigma \in \mathcal{T}} \E_\Q\left[e^{(1-a) L_\sigma - U_\sigma}\oo_{\{\sigma<\tau_\infty\} \cap \{1/Z_\sigma \le 1/\kappa_1\}}\right]  = \sup_{\sigma \in \mathcal{T}} \E_\P\left[e^{B^a_\sigma - U_\sigma}\oo_{\{\sigma<\tau_0\} \cap \{Z_\sigma\ge \kappa_1\}}\right] < \varepsilon,
\end{equation}
using again Lemma~\ref{L:ABdecomp}. Now, it follows from Corollary~\ref{C:conv001} with $X=L$ and $f(x) = e^{(1-a)x}$ for all $x\in\R$ that~$L$ converges under~$\Q$. Hence $\inf_{t<\tau} ((1-a)L_t-U_t) = -\Theta$ for some finite nonnegative random variable $\Theta$. Furthermore, since by assumption $N$ does not converge under $\Q$, there is an event $C$ with $\Q(C)>0$ such that $\sigma<\tau$ on $C$, where $\sigma=\inf\{t\geq 0:1/Z_t\le1/\kappa\}$. Consequently, the left-hand side of \eqref{eq:C:NK1:001} is bounded below by $\E_\Q[e^{-\Theta}\oo_C]>0$, independently of~$\varepsilon$. This gives the desired contradiction.

Finally, \ref{C:NK1:6} implies that $(1+x)\log(1+x) \oo_{x > \kappa} * \nu^M$ is finite-valued. Hence by Corollary~\ref{C:convYX2} it also implies the remaining conditions, provided $a\le0$.
\end{proof}

\begin{remark}
Without the assumption that $(1+x)\log(1+x) \oo_{x > \kappa} * \nu^M$ is finite-valued for some $\kappa>0$, the conditions \ref{C:NK1:4} and \ref{C:NK1:5}  in Corollary~\ref{C:NK1} can be satisfied for all $a>1$ even if $Z$ is not a uniformly integrable martingale. On the other hand, there exist uniformly integrable martingales $Z$ for which \eqref{T:NK:eq1} does not hold for all $a<1$. These points are illustrated in Example~\ref{ex:5 16} below.
\qed
\end{remark}

The implications Theorem~\ref{T:NK}\ref{T:NK:1} $\Longrightarrow$ Corollary~\ref{C:NK1}\ref{C:NK1:2} and Theorem~\ref{T:NK}\ref{T:NK:3} $\Longrightarrow$ Corollary~\ref{C:NK1}\ref{C:NK1:4}--\ref{C:NK1:6} seem to be new, even in the continuous case. The reverse directions
 imply several well-known criteria in the literature. Setting $a=0$ and $U=0$ in~\eqref{eq:NPN}, and using that $A^0$ is nondecreasing, we recover the first condition in Theorem~\ref{T:LepMem}. More generally, taking $a\in[0,1)$ and $U=0$ yields a strengthening of Theorem~I.1(5-$\alpha$) in \citet{Lepingle_Memin_integrabilite}. 
Indeed, the $L^1$-boundedness in~\eqref{T:NK:eq1}, rather than uniform integrability as assumed by L\'epingle and M\'emin (with $U=0$), suffices to conclude that $Z$ is a uniformly integrable martingale. This is however not the case when $A^a$ is replaced by $B^a$; the uniform integrability assumption in \ref{C:NK1:5} cannot the weakened to $L^1$-boundedness in general. Counterexamples to this effect are constructed in Subsection~\ref{SS:counter NK}. 
However,  the implication Corollary~\ref{C:NK1}\ref{C:NK1:4} $\Longrightarrow$ Theorem~\ref{T:NK}\ref{T:NK:3}, which also seems to be new, shows that if the jumps of $M$ are bounded away from zero then uniform integrability can be replaced by $L^1$--boundedness. 

In a certain sense our results quantify how far the L\'epingle and M\'emin conditions are from being necessary: the gap is precisely captured by the extended locally integrable (under $\Q$) process~$U$. In practice it is not clear how to find a suitable process $U$. A natural question is therefore how well one can do by restricting to the case $U=0$. Theorem~\ref{T:NK} suggests that one should look for other functions $f$ than the exponentials chosen in Corollary~\ref{C:NK1}. The best possible choice is $f(x)=x\oo_{x>\kappa}$ for some $\kappa>0$, which, focusing on $\epsilon=1$ and $A=A^0$ for concreteness, leads to the criterion
\[
\sup_{\sigma\in\Tcal} \E_\P\left[ e^{A_\sigma} N_\sigma e^{-N_\sigma}\1{N_\sigma>\kappa}\right] <\infty.
\]
Here one only needs to control $A$ on the set where $N$ takes large positive values. Moreover, on this set one is helped by the presence of the small term $N e^{-N}$.

We now state a number a further consequences of the above results, which complement and improve various criteria that have already appeared in the literature.  Again we refer the reader to Remark~\ref{R:ELI Q P} for an important comment on the extended local integrability assumptions appearing below.

\begin{corollary}[Kazamaki type conditions] \label{C:nonpred1}
Each of the following conditions implies that $Z$ is a uniformly integrable martingale:
\begin{enumerate}
\item\label{C:nonpred1:1} The running supremum of ${A}^a$ is extended locally integrable on~$\lc0,\tau_\infty\lc$ under~$\Q$ for some $a\ne1$.
\item\label{C:nonpred1:1'} The running supremum of ${B}^a$ is extended locally integrable on~$\lc0,\tau_\infty\lc$ under~$\Q$ for some $a\ne1$.
\item\label{C:nonpred1:2} ${\displaystyle \sup_{\sigma \in \mathcal{T}} \E_\P\left[\exp\left( \frac{1}{2} M_\sigma +  \left(   \log(1+x) - \frac{x^2 +2 x}{2(1+x)}\right) \oo_{x<0} * \mu^M_\sigma \right) \1{\sigma<\tau_0} \right] < \infty.}$
\item\label{C:nonpred1:2'} ${\displaystyle \left(\exp\left( \frac{1}{2} M_\sigma +  \frac{1}{2}\left(  (1+x)  \log(1+x) - x\right) * \nu^M_\sigma \right) \1{\sigma<\tau_0} \right)_{\sigma \in \mathcal{T}}}$ is uniformly integrable.
\item\label{C:nonpred1:3} $M$ is a uniformly integrable martingale and
\begin{align}  
\E_\P\left[\exp\left( \frac{1}{2} M_{\tau_0-} +  \left(   \log(1+x) - \frac{x^2 +2 x}{2(1+x)}\right) \oo_{x<0} * \mu^M_{\tau_0-} \right) \right] &< \infty.  \label{eq:NPN3b}
\end{align}
\item\label{C:nonpred1:3'} $M$ is a uniformly integrable martingale and
\begin{align*}  
\E_\P\left[\exp\left( \frac{1}{2} M_{\tau_0-} +   \frac{1}{2}\left(  (1+x)  \log(1+x) - x\right) * \nu^M_{\tau_0-} \right) \right] &< \infty. 
\end{align*}
\item\label{C:nonpred1:4} $M$ satisfies $\Delta M \geq  -1+\delta$ for some $\delta > 0$ and
\begin{align}  \label{eq:NPN4}
\sup_{\sigma \in \mathcal{T}} \E_\P\left[ \exp\left(\frac{M_\sigma}{1+\delta} - \frac{1-\delta}{2 + 2 \delta} [M^c,M^c]\right)  \1{\sigma<\tau_0} \right] < \infty.
\end{align}
\end{enumerate}
\end{corollary}

\begin{proof}
For~\ref{C:nonpred1:1} and~\ref{C:nonpred1:1'}, take $U=A^a$ and $U=B^a$ in Corollary~\ref{C:NK1}\ref{C:NK1:2} and~\ref{C:NK1:5}, respectively. For~\ref{C:nonpred1:2} and~\ref{C:nonpred1:2'}, take $U=0$ and $a=1/2$ in Corollary~\ref{C:NK1}\ref{C:NK1:2} and~\ref{C:NK1:5}, and use the inequalities $\log(1+x) \leq (x^2 + 2x)/(2 +2x)$ for all $x \geq 0$. For~\ref{C:nonpred1:3} and~\ref{C:nonpred1:3'}, note that if $M$ is a uniformly integrable martingale then the exponential processes in~\ref{C:nonpred1:2} and~\ref{C:nonpred1:2'} are submartingales, thanks to the inequality $(1+x)\log(1+x)-x\ge0$ for all $x > -1$. Thus~\ref{C:nonpred1:3} implies that~\ref{C:nonpred1:2} holds, and~\ref{C:nonpred1:3'} implies that~\ref{C:nonpred1:2'} holds. Finally, due to the inequality
\[
\log(1+x) - \frac{x^2/(1+\delta) + x}{1+x} = \frac{1}{1+\delta} \int_0^x \frac{-y}{(1+y)^2} (1-\delta+y) \dd y \ge 0
\]
for all $x \ge -1+\delta$, \ref{C:nonpred1:4} implies that Corollary~\ref{C:NK1}\ref{C:NK1:2} holds with $a=1/(1+\delta)$ and $U=0$.
\end{proof}

\begin{remark}
We make the following observations concerning Corollary~\ref{C:nonpred1}:
\begin{itemize}
\item The condition in Corollary~\ref{C:nonpred1}\ref{C:nonpred1:1} is sufficient but not necessary for the conclusion, as Example~\ref{ex:5 16} below illustrates. Similarly, it can be shown that the condition in Corollary~\ref{C:nonpred1}\ref{C:nonpred1:1'} is not necessary for $Z$ to be a uniformly integrable. However, the condition in Theorem~\ref{T:NK}\ref{T:NK:3} implies that $B^a$ is extended locally integrable on~$\lc0,\tau_\infty\lc$ under~$\Q$. In view of the ``code book'', this can be seen by piecing together Lemmas~\ref{L:convYX} and~\ref{L:ABdecomp}, Theorem~\ref{T:convYX}, and~\eqref{eq:VV}.
\item The uniform integrability of $M$ is needed to argue that \eqref{eq:NPN3b} implies~\ref{C:nonpred1:2} in Corollary~\ref{C:nonpred1}. Even if $M$ is continuous this implication is false in general; see \citet{Ruf_Novikov} for examples.
\qed
\end{itemize}
\end{remark}

Corollary~\ref{C:nonpred1}\ref{C:nonpred1:3} appears already in   Proposition~I.3 in \citet{Lepingle_Memin_integrabilite}. Corollary~\ref{C:nonpred1}\ref{C:nonpred1:4} with the additional assumption that  $\delta \leq 1$ implies Proposition~I.6 in \citet{Lepingle_Memin_integrabilite}.   Also, conditions \ref{C:nonpred1:3}, \ref{C:nonpred1:3'} and (a somewhat weaker version of) \ref{C:nonpred1:4} below have appeared in \cite{Yan_1982}.    In particular, if $\Delta M \geq 0$ and $\delta = 1$, \eqref{eq:NPN4} yields Kazamaki's condition verbatim.

\begin{example} \label{ex:5 16}
Let $Y$ be a nonnegative random variable such that $\E_\P[Y]<\infty$ and $\E_\P[(1+Y)\log(1+Y)]=\infty$, let $\Theta$ be a $\{0,1\}$--valued random variable with $\P(\Theta=1)=1/(1+2\E_\P[Y])$, and let $W$ be standard Brownian motion. Suppose $Y$, $\Theta$, $W$ are pairwise independent. Now define 
\[
M = \left(Y\Theta - \frac{1}{2}(1-\Theta) + W - W_1\right)\oo_{\lc 1, \infty\lc}.
\]
Then $M$ is a martingale under its natural filtration with $\Delta M\ge-1/2$, and the process $Z=\Ecal(M)$ is not a uniformly integrable martingale, as it tends to zero as $t$ tends to infinity. However,
\[
((1+x)\log(1+x) - x) * \nu^M_1=\E_\P[(1+Y)\log(1+Y)-Y]=\infty,
\]
which implies that conditions \ref{C:NK1:4}--\ref{C:NK1:5} in Corollary~\ref{C:NK1} are satisfied for any $a>1$, apart from the finiteness of $(1+x)\log(1+x) \oo_{x\geq \kappa}* \nu^M$ for some $\kappa>0$.

Consider now the process $\widetilde Z=(Z_{t\wedge1})_{t\ge0} = (Y\Theta + \frac{1}{2}(1+\Theta)) \oo_{\lc 1, \infty\lc}$. This is a uniformly integrable martingale. Nonetheless, \eqref{eq:4.7} fails for any $a<1$. We now consider the process
$$\widetilde A^a = \left(\log(1+\Delta M_1) - (1-a) \frac{\Delta M_1}{1+\Delta M_1}\right) \oo_{\lc 1, \infty\lc}$$ for each $a \in \R$ as in \eqref{eq:A}. Then $\E_\P[\widetilde A_1^a \widetilde Z_1] = \infty$, which implies that $\widetilde A^a$ is not extended locally integrable under~$\Q$ for each $a \in \R$, as can be deduced based on Lemma~\ref{L:extended locally}. In particular, the condition in Corollary~\ref{C:nonpred1}\ref{C:nonpred1:1} is not satisfied.
\qed
\end{example}

\subsection{Further characterizations}

We now present a number of other criteria that result from our previous analysis, most of which seem to be new. Again, the reader should keep Remark~\ref{R:ELI Q P} in mind.

\begin{theorem}[Necessary and sufficient conditions based on extended localization]\label{T:further}
Let  $\epsilon \in \{-1,1\}$, $\eta \in (0,1)$, and $\kappa>0$.
Then the following  conditions are equivalent:
\begin{enumerate}[label={\rm(\alph*)},ref={\rm(\alph*)}]
	\item\label{T:further:1} $Z$ is a uniformly integrable martingale.
	\item\label{T:further:2} $(\epsilon N)^+$ is extended locally integrable on~$\lc0,\tau_\infty\lc$   under $\Q$.
	\item\label{T:further:3} $[M^c, M^c]  + (x^2 \wedge |x|) * \nu^M$ is extended locally integrable  on~$\lc0,\tau_\infty\lc$  under $\Q$.
\end{enumerate}
Moreover, the following two conditions are equivalent:
\begin{enumerate}[label={\rm(\alph*)},ref={\rm(\alph*)}, resume]
	\item\label{T:further:4} $Z$ is a uniformly integrable martingale and $((\Delta M)^- / (1+\Delta M))^2$ is extended locally integrable on~$\lc0,\tau_\infty\lc$  under $\Q$.
	\item\label{T:further:5} ${\displaystyle [M^c, M^c]  + {(x/(1+x))^2} * \mu^M}$ is extended locally integrable on~$\lc0,\tau_\infty\lc$  under $\Q$.
\end{enumerate}
Furthermore, the following conditions are equivalent:
\begin{enumerate}[label={\rm(\alph*)},ref={\rm(\alph*)},resume]
\item\label{T:further:6}  $Z$ is a uniformly integrable martingale and  $(1+x)\log(1+x) \oo_{x>\kappa} * \nu^M$ is extended locally integrable on~$\lc0,\tau_\infty\lc$  under $\Q$.
\item\label{T:further:7} ${\displaystyle  [M^c,M^c] + ((1+x)\log(1+x)-x) * \nu^M}$ is extended locally integrable on~$\lc0,\tau_\infty\lc$  under $\Q$.
\item\label{T:further:8}  $(1+x)\log(1+x) \oo_{x>\kappa} * \nu^M$ is finite-valued, $-x \oo_{x<-\eta} * \nu^M$ is extended locally integrable on~$\lc0,\tau_\infty\lc$  under $\Q$, and $(\epsilon L)^+$ is extended locally integrable under $\Q$.
\end{enumerate}
\end{theorem}

\begin{proof}[Proof of Theorem~\ref{T:further}]
Once again we use the ``code book'' from Subsection~\ref{S:method} freely. A calculation using Theorem~\ref{T:reciprocal} yields
\[
[M^c, M^c]  + (x^2 \wedge |x|) * \nu^M = \langle N^c,N^c\rangle + \left(\frac{y^2}{1+y}\wedge|y|\right)*\widehat\nu^N.
\]
The equivalence of \ref{T:further:1}--\ref{T:further:3} now follows from Theorem~\ref{T:conv} using the inequalities $(y^2\wedge|y|)/2 \le (y^2/(1+y))\wedge|y| \le 2(y^2\wedge|y|)$.
 
Since $(\Delta N)^+=(\Delta M)^-/(1+\Delta M)$ and $\Delta N>-1$, \ref{T:further:4} holds if and only if $N$ converges and $(\Delta N)^2$ is extended locally integrable under~$\Q$. By Corollary~\ref{C:convergence_QV} and Lemma~\ref{L:ELI}\ref{L:ELI:3}, this holds if and only if $[N,N]$ is extended locally integrable under~$\Q$. Since $[N,N]=\langle M^c,M^c\rangle+(x/(1+x))^2*\mu^M$, this is equivalent to \ref{T:further:5}.

To prove the equivalence of \ref{T:further:6}--\ref{T:further:8}, first note that $((1+x)\log(1+x) - x) * \nu^M$ is finite-valued under either condition, and hence so is~$L$. Note also the equalities
\begin{align*}
(1+x)\log(1+x)\oo_{x\ge \kappa}*\nu^M  &=-\log(1+y)\oo_{y\le-\kappa/(1+\kappa)}*\widehat\nu^N;\\
 [M^c,M^c] + ((1+x)\log(1+x)-x) * \nu^M &= [N^c,N^c] + (y-\log(1+y)) * \widehat\nu^N;\\
-x\oo_{x<-\eta}*\nu^M &= y\oo_{y>\eta/(1-\eta)}*\widehat\nu^N.
\end{align*}
With the identifications in the ``code book'', \ref{T:further:6} now states that the event~\eqref{T:convYX:3} has full probability under~$\Q$. Moreover, \ref{T:further:7} states that the event~\eqref{T:convYX:2} has full probability under $\Q$. Thanks to Lemma~\ref{L:convYX},  \ref{T:further:8} states that \eqref{T:convYX:4} has full probability under $\Q$. Thus all three conditions are equivalent by Theorem~\ref{T:convYX}.
\end{proof}

\begin{remark}
We make the following observations concerning Theorem~\ref{T:further}:
\begin{itemize}
\item If the jumps of $M$ are bounded away from $-1$, that is, $\Delta M \geq -1+\delta$ for some $\delta > 0$, then the second condition in Theorem~\ref{T:further}\ref{T:further:4} is automatically satisfied.
\item If $[M^c, M^c]_{\tau_0-}  + (x^2 \wedge |x|) * \nu^M_{\tau_0-} $ is extended locally integrable then $\tau_0 = \infty$ and $Z_\infty>0$ by Theorem~\ref{T:conv}. Contrast this with the condition in Theorem~\ref{T:further}\ref{T:further:3}.
\qed
\end{itemize}
\end{remark}

The implication \ref{T:further:3} $\Longrightarrow$ \ref{T:further:1} of Theorem~\ref{T:further} is proven in Theorem~12 in \cite{Kabanov/Liptser/Shiryaev:1979} if the process  in Theorem~\ref{T:further}\ref{T:further:3} is not only extended locally integrable, but bounded.

\section{Counterexamples}  \label{S:examp}

In this section we collect several examples of local martingales that illustrate the wide range of asymptotic behavior that can occur. This showcases the sharpness of the results in Section~\ref{S:convergence}. In particular, we focus on the role of the extended uniform integrability of the jumps.

\subsection{Random walk with large jumps}  \label{A:SS:lack}

Choose a sequence $(p_n)_{n \in \N}$ of real numbers such that $p_n\in(0,1)$ and $\sum_{n =1}^\infty p_n < \infty$. Moreover, choose a sequence $(x_n)_{n \in \N}$ of real numbers.
Then let $(\Theta_n)_{n \in \N}$ be a sequence of independent random variables with $\P(\Theta_n = 1) = p_n$ and $\P(\Theta_n = 0) = 1-p_n$ for all $n \in N$.  Now define a process $X$ by
\begin{align*}
	X_t = \sum_{n =1}^{[t]} x_n \left(1 - \frac{\Theta_n}{p_n}\right),
\end{align*}
where $[t]$ is the largest integer less  than or equal to  $t$, and let $\F$ be its natural filtration. Clearly $X$ is a locally bounded martingale. The Borel-Cantelli lemma implies that  $\Theta_n$ is nonzero for only finitely many $n \in \N$, almost surely, whence for all sufficiently large $n \in \N$ we have $\Delta X_n = x_n$. By choosing a suitable sequence $(x_n)_{n \in \N}$ one may therefore achieve essentially arbitrary asymptotic behavior. This construction was inspired by an example due to George Lowther that appeared on his blog Almost Sure on December 20, 2009.

\begin{lemma} \label{L:ex1}
With the notation of this subsection, $X$ satisfies the following properties:
\begin{enumerate}
\item\label{L:ex1:1} $\lim_{t \to \infty}  X_t$ exists in $\R$ if and only if $\lim_{m\to \infty} \sum_{n = 1}^m x_n$ exists in $\R$.
\item\label{L:ex1:3} $(1 \wedge x^2) * \mu^X_{\infty-} < \infty$ if and only if $[X,X]_{\infty-} < \infty$ if and only if $\sum_{n = 1}^\infty x_n^2 < \infty$.
\item\label{L:ex1:2} $X$ is a semimartingale on $[0,\infty]$ if and only if $(x^2\wedge|x|)*\nu^X_{\infty-}<\infty$ if and only if $X$ is a uniformly integrable martingale if and only if  $\sum_{n = 1}^\infty |x_n| < \infty$.
\end{enumerate}
\end{lemma}
\begin{proof}
The statements in \ref{L:ex1:1} and \ref{L:ex1:3} follow from the Borel-Cantelli lemma. For \ref{L:ex1:2}, note that $|X_t| \le \sum_{n\in\N} |x_n|  (1 + \Theta_n/p_n)$ for all $t\ge0$. Since
\[
\E\left[\sum_{n = 1}^\infty |x_n|  \left(1 + \frac{\Theta_n}{p_n}\right)\right]=2\sum_{n = 1}^\infty |x_n|,
\]
the condition $\sum_{n = 1}^\infty |x_n| < \infty$ implies that $X$ is a uniformly integrable martingale, which implies that~$X$ is a special semimartingale on $[0,\infty]$, or equivalently that $(x^2\wedge|x|)*\nu^X_{\infty-}<\infty$ (see Proposition~II.2.29 in~\cite{JacodS}), which implies that $X$ is a semimartingale on $[0,\infty]$. It remains to show that this implies $\sum_{n = 1}^\infty |x_n| < \infty$. We prove the contrapositive, and assume $\sum_{n = 1}^\infty |x_n| = \infty$. Consider the bounded predictable process $H = \sum_{n=1}^\infty (\oo_{x_n>0} - \oo_{x_n<0}) \oo_{\lc n\rc}$. If $X$ were a semimartingale on $[0,\infty]$, then $(H\cdot X)_{\infty-}$ would be well-defined and finite. However, by Borel-Cantelli, $H\cdot X$ has the same asymptotic behavior as $\sum_{n=1}^\infty |x_n|$ and thus diverges. Hence $X$ is not a semimartingale on~$[0,\infty]$.
\end{proof}

Martingales of the above type can be used to illustrate that much of Theorem~\ref{T:conv} and its corollaries fails if one drops extended local integrability of $(\Delta X)^-\wedge X^-$. We now list several such counterexamples.

\begin{example} \label{E:P1}
We use the notation of this subsection.
\begin{enumerate}
\item \label{E:P1:1} Let $x_n = (-1)^n/\sqrt{n}$ for all $n \in \N$. Then
\[
\P\Big(\lim_{t \to \infty} X_t \text{ exists in $\R$}\Big)  = \P\left([X,X]_{\infty-} = \infty\right) =  \P\left(x^2 \oo_{|x|<1}* \nu^X_{\infty-} = \infty\right) = 1.
\]
Thus the implications \ref{T:conv:a} $\Longrightarrow$ \ref{T:conv:c} and \ref{T:conv:a} $\Longrightarrow$ \ref{T:conv:f} in Theorem~\ref{T:conv} fail without the integrability condition on $(\Delta X)^-\wedge X^-$. Furthermore, by setting $x_1=0$ but leaving  $x_n$ for all $n \geq 2$ unchanged, and ensuring that $p_n\ne x_n/(1+x_n)$ for all $n \in \N$, we have $\Delta X\ne-1$. Thus, $\Ecal(X)_t=\prod_{n=1}^{[t]} (1+\Delta X_n)$ is nonzero for all $t$. Since $\Delta X_n=x_n$ for all sufficiently large $n \in \N$, $\Ecal(X)$ will eventually be of constant sign. Moreover, for any $n_0\in\N$ we have
\[
\lim_{m\to\infty}\sum_{n=n_0}^m \log(1+x_n)\le \lim_{m\to\infty}\sum_{n=n_0}^m \left(x_n - \frac{x_n^2}{4}\right)=-\infty.
\]
It follows that $\P( \lim_{t\to\infty}\Ecal(X)_t=0) = 1$, showing that the implication \ref{T:conv:a} $\Longrightarrow$ \ref{T:conv:g} in Theorem~\ref{T:conv} fails without the integrability condition on $(\Delta X)^-\wedge X^-$.

\item  Part~\ref{E:P1:1} illustrates that the implications \ref{T:conv:a'} $\Longrightarrow$ \ref{T:conv:c},  \ref{T:conv:a'} $\Longrightarrow$ \ref{T:conv:f}, and \ref{T:conv:a'} $\Longrightarrow$ \ref{T:conv:g} in Theorem~\ref{T:conv} fail without the integrability condition on $(\Delta X)^-\wedge X^-$.
We now let $x_n = 1$ for all $n \in \N$. Then $\P(\lim_{t \to \infty} X_t = \infty)  = 1$, which illustrates that also \ref{T:conv:a'} $\Longrightarrow$ \ref{T:conv:a} in that theorem fails without integrability condition.

\item We now fix a sequence $(x_n)_{n \in \N}$ such that $|x_n| = 1/n$  but $g: m \mapsto \sum_{n = 1}^m x_n$ oscillates with $\liminf_{m \to \infty} g(m) = -\infty$ and  $\limsup_{m \to \infty} g(m) = \infty$.  This setup illustrates that \ref{T:conv:f} $\Longrightarrow$ \ref{T:conv:a} and \ref{T:conv:f} $\Longrightarrow$ \ref{T:conv:a'}  in Theorem~\ref{T:conv}  fail without the integrability condition on $(\Delta X)^- \wedge X^-$.  Moreover, by Lemma~\ref{L:ex1}\ref{L:ex1:2} the implication \ref{T:conv:f} $\Longrightarrow$ \ref{T:conv:c} fails without the additional integrability condition. The same is true for the implication \ref{T:conv:f} $\Longrightarrow$ \ref{T:conv:g}, since $\log \mathcal{E}(X) \leq X$.

\item Let $x_n=e^{(-1)^n/\sqrt{n}}-1$ and suppose $p_n\ne x_n/(1+x_n)$ for all $n \in \N$ to ensure $\Delta X\ne-1$. Then
\[
\P\Big( \lim_{t\to\infty}\Ecal(X)_t \text{ exists in }\R\setminus\{0\}\Big) = \P\Big(\lim_{t \to \infty} X_t = \infty\Big) = \P\Big([X,X]_{\infty-}= \infty\Big) = 1.
\]
Indeed, $\lim_{m\to\infty}\sum_{n=1}^m \log(1+x_n)=\lim_{m\to\infty}\sum_{n=1}^m (-1)^n/\sqrt{n}$ exists in $\R$, implying that $\Ecal(X)$ converges to a nonzero limit. Moreover,
\[
\lim_{m\to\infty}\sum_{n=1}^m x_n \ge \lim_{m\to\infty}\sum_{n=1}^m \left(\frac{(-1)^n}{\sqrt{n}}+\frac{1}{4n}\right)=\infty,
\]
whence $X$ diverges. Since $\sum_{n=1}^\infty x_n^2 \ge \sum_{n=1}^\infty 1/(4n)=\infty$, we obtain that $[X,X]$ also diverges. Thus the implications \ref{T:conv:g} $\Longrightarrow$ \ref{T:conv:a} and \ref{T:conv:g} $\Longrightarrow$ \ref{T:conv:f} in Theorem~\ref{T:conv} fail without the integrability condition on $(\Delta X)^-\wedge X^-$. So does the implication \ref{T:conv:g} $\Longrightarrow$ \ref{T:conv:c} due to Lemma~\ref{L:ex1}\ref{L:ex1:2}. Finally, note that the implication \ref{T:conv:g} $\Longrightarrow$ \ref{T:conv:a'} holds independently of any integrability conditions since $\log \mathcal{E}(X) \leq X$.

\item Let $x_n=-1/n$ for all $n\in\N$. Then $[X,X]_{\infty-}<\infty$ and $(\Delta X)^-$ is extended locally integrable, but $\lim_{t\to\infty}X_t=-\infty$. This shows that the condition involving limit superior is needed in Theorem~\ref{T:conv}\ref{T:conv:f}, even if $X$ is a martingale. We further note that if $X$ is Brownian motion, then $\limsup_{t\to\infty}X_t>-\infty$ and $(\Delta X)^-=0$, but $[X,X]_{\infty-}=\infty$. Thus some condition involving the quadratic variation is also needed in Theorem~\ref{T:conv}\ref{T:conv:f}.

\item Note that choosing $x_n= (-1)^n/n$ for each $n\in \N$ yields a locally bounded martingale~$X$ with $[X,X]_\infty<\infty$, $X_\infty = \lim_{t \to \infty} X_t$ exists, but $X$ is not a semimartingale on $[0,\infty]$.   This contradicts statements in the literature which assert that a semimartingale that has a limit is a semimartingale on the extended interval.  This example also illustrates that the implications \ref{T:conv1:a} $\Longrightarrow$ \ref{C:conv001:d} and \ref{T:conv1:a} $\Longrightarrow$ \ref{C:conv001:e}
in Corollary~\ref{C:conv001} fail without additional integrability condition. For the sake of completeness, Example~\ref{ex:semimartingale} illustrates that the integrability condition in Corollary~\ref{C:conv001}\ref{C:conv001:e} is not redundant either. \qed
\end{enumerate}
\end{example}

\begin{remark}
Many other types of behavior can be generated within the setup of this subsection. For example, by choosing the sequence $(x_n)_{n \in \N}$ appropriately we can obtain a martingale $X$ that converges nowhere, but satisfies $\P(\sup_{t\geq 0} |X_t| < \infty)=1$. We can also choose $(x_n)_{n \in \N}$ so that, additionally, either $\P([X,X]_{\infty-} = \infty)=1$ or $\P([X,X]_{\infty-} <\infty)=1$.
\qed
\end{remark}

\begin{example} \label{ex:ui}
The uniform integrability assumption in Corollary~\ref{C:convYX2} cannot be weakened to $L^1$--boundedness. To see this, within the setup of this subsection, let $x_n=-1/2$. Then $\Delta X\ge-1/2$. Moreover,  the sequence $(p_n)_{n\in\N}$ can be chosen so that
\begin{align}  \label{ex:ui:eq1}
	\sup_{\sigma\in\Tcal} \E\left[ e^{c\log(1+x) * (\mu^X-\nu^X)_\sigma}  \right] < \infty
\end{align}
for each $c<1$,
while, clearly, $\P(\lim_{t \to \infty} X_t = -\infty)=1$. This shows that the implication \ref{T:convYX2:b} $\Longrightarrow$ \ref{T:convYX2:a} in Corollary~\ref{C:convYX2}, with $c<1$, fails without the tail condition on $\nu^X$.

To obtain~\eqref{ex:ui:eq1}, note that $Y=\log(1+x) * (\mu^X-\nu^X)$ is a martingale, so that $e^{cY}$ is a submartingale, whence $\E[e^{cY_\sigma}]$ is nondecreasing in~$\sigma$. Since the jumps of $X$ are independent, this yields
\begin{align*}
	\sup_{\sigma\in\Tcal}  \E\left[ e^{c \log(1+x) * (\mu^X-\nu^X)_\sigma}  \right]
\le  \prod_{n=1}^\infty \E\left[ (1+\Delta X_n)^c \right] e^{-c\,\E[\log(1+\Delta X_n)]} =: e^{\kappa_n}.
\end{align*}
We have $\kappa_n\ge0$ by Jensen's inequality, and a direct calculation yields
\begin{align*}
\kappa_n =& \log\E\left[ (1+\Delta X_n)^{c} \right] - c\,\E[\log(1+\Delta X_n)] 
\le \log\left( 2 p_n^{1-c} +1\right) - c\,p_n \log(1+p_n^{-1})
\end{align*}
for all $c <1$.
Let us now fix a sequence $(p_n)_{n \in \N}$ such that the following inequalities hold for all $n \in \N$:
\begin{align*}
	p_n \log(1+p_n^{-1})  \leq \frac{1}{n^3} \qquad \text{and} \qquad
	p_n \leq \frac{1}{2^n}\left( e^{n^{-2}}  -1\right)^n. 
\end{align*}
This is always possible. Such a sequence satisfies $\sum_{n\in\N}p_n<\infty$ and results in $\kappa_n \leq 2/n^2$ for all $n \geq -c \vee (1/(1-c))$, whence $\sum_{n\in\N}\kappa_n<\infty$. This yields the assertion.
\qed
\end{example}

\subsection{Quasi-left continuous one-jump martingales}  \label{A:SS:one}

We now present examples based on a martingale $X$ which, unlike in Subsection~\ref{A:SS:lack}, has one single jump that occurs at a totally inaccessible stopping time. In particular, the findings of Subsection~\ref{A:SS:lack} do not rely on the fact that the jump times there are predictable.

Let $\lambda, \gamma:\mathbb R_+\to\mathbb R_+$ be two continuous nonnegative functions. Let $\Theta$ be a standard exponential random variable and define $\rho = \inf\{t\ge 0: \int_0^t\lambda(s) ds \ge \Theta\}$. Let $\F$ be the filtration generated by the indicator process $\oo_{\lc\rho,\infty\lc}$, and define a process $X$ by
\[
X_t = \gamma(\rho)\1{\rho\le t} - \int_0^t \gamma(s)\lambda(s)\1{s<\rho}\dd s.
\]
Note that $X$ is the integral of $\gamma$ with respect to $\1{\rho\le t}-\int_0^{t\wedge\rho}\lambda_s\dd s$ and is  a martingale. Furthermore, $\rho$ is totally inaccessible. This construction is sometimes called the {\em Cox construction}. Furthermore, the jump measure $\mu^X$ and corresponding compensator $\nu^X$ satisfy
\[
F*\mu^X = F(\rho,\gamma(\rho))\oo_{\lc \rho, \infty\lc}, \qquad
F*\nu^X_t = \int_0^{t\wedge\rho}F(s,\gamma(s))\lambda(s) \dd s
\]
for all $t \geq 0$,
where $F$ is any nonnegative predictable function. We will study such martingales when $\lambda$ and $\gamma$ posses certain integrability properties, such as the following:
\begin{align}
&\int_0^\infty\lambda(s)\dd s <\infty;   \label{A:eq:int1}\\
&\int_0^\infty\gamma(s)\lambda(s)\dd s =\infty;  \label{A:eq:int2}\\
&\int_0^\infty (1+\gamma(s))^c\lambda(s) \dd s <\infty \quad \text{ for all }c<1.  \label{A:eq:int3}
\end{align}
For instance, $\lambda(s) = 1/(1+s)^2$ and $\gamma(s)=s$ satisfy all three properties.

\begin{example} \label{E:2}
The limit superior condition in Theorem~\ref{T:conv} is essential, even if $X$ is a local martingale. Indeed, with the notation of this subsection, let $\lambda$ and $\gamma$ satisfy \eqref{A:eq:int1} and \eqref{A:eq:int2}. Then
\begin{align*}
&\P\left([X,X]_{\infty-} +  (x^2\wedge 1) * \nu^X_{\infty-}  < \infty\right) = \P\Big(\sup_{t\ge 0} X_t < \infty\Big)=1; \\
&\P \Big(\limsup_{t\to \infty} X_t = - \infty\Big) > 0.
\end{align*}
This shows that finite quadratic variation does not prevent a martingale from diverging; in fact, $X$ satisfies $\{[X,X]_{\infty-} =0\} = \{\limsup_{t\to \infty}X_t = - \infty\}$. The example also shows that one cannot replace $(x^2\wedge |x|)*\nu^X_{\infty-}$ by $(x^2\wedge 1) * \nu^X_{\infty-}$ in \eqref{T:conv2:2}. Finally, it illustrates in the quasi-left continuous case that diverging local martingales need not oscillate, in contrast to continuous local martingales.

To prove the above claims, first observe that $[X,X]_{\infty-} = \gamma(\rho)^2 \1{\rho<\infty} < \infty$ and $\sup_{t\ge 0} X_t\le\gamma(\rho)\1{\rho<\infty}<\infty$ almost surely. Next, we get $\P(\rho=\infty) = \exp({-\int_0^\infty\lambda(s) \dd s})>0$ in view of~\eqref{A:eq:int1}. We conclude by observing that $\lim_{t \to \infty} X_t = - \lim_{t \to \infty}  \int_0^t \gamma(s)\lambda(s) \dd s = -\infty$ on the event $\{\rho=\infty\}$ due to~\eqref{A:eq:int2}.
\qed
\end{example}

\begin{example}
Example~\ref{E:2} can be refined to yield a martingale with a single positive jumps, that diverges without oscillating, but has infinite quadratic variation. To this end, extend the probability space to include a Brownian motion $B$ that is independent of~$\Theta$, and suppose $\F$ is generated by $(\oo_{\lc\rho,\infty\lc},B)$. The construction of $X$ is unaffected by this. In addition to \eqref{A:eq:int1} and \eqref{A:eq:int2}, let $\lambda$ and $\gamma$ satisfy
\begin{equation} \label{eq:E:3:1}
\lim_{t\to\infty} \frac{\int_0^t \gamma(s)\lambda(s) \dd s }{\sqrt{2 t \log \log t}} = \infty.
\end{equation}
For instance, take $\lambda(s)=1/(1+s)^2$ and $\gamma(s)=1/\lambda(s)$. Then the martingale $X' = B + X$ satisfies
\begin{equation} \label{eq:E:3:2}
\P\left([X',X']_{\infty-} = \infty\right)= 1 \qquad\text{and}\qquad \P\Big(\sup_{t\ge 0} X'_t < \infty\Big) > 0,
\end{equation}
so that, in particular, the inclusion $\{ [X',X']_\infty=\infty\} \subset \{\sup_{t\ge 0} X'_t=\infty\}$ does not hold in general. 

To prove~\eqref{eq:E:3:2}, first note that $[X',X']_{\infty-} \geq [B,B]_{\infty-} = \infty$. Next, \eqref{eq:E:3:1} and the law of the iterated logarithm yield, on the event $\{\rho=\infty\}$,
\[
\limsup_{t \to \infty} X'_t = \limsup_{t \to \infty} \Big(B_t - \int_0^t \gamma(s)\lambda(s) \dd s\Big)  \leq \limsup_{t \to \infty} \Big(2 \sqrt{2 t \log \log t}  - \int_0^t \gamma(s)\lambda(s) \dd s\Big) = -\infty.
\]
Since $\P(\rho=\infty)>0$, this implies $\P(\sup_{t\geq 0} X'_t < \infty)>0$.
\qed
\end{example}

\begin{example} \label{ex:semimartingale}
The semimartingale property does not imply that $X^- \wedge (\Delta X)^-$ is extended local integrability. With the notation of this subsection, consider the process $\widehat{X} = -\gamma(\rho) \oo_{\lc \rho, \infty\lc}$, which is clearly a semimartingale on $[0,\infty]$. On $[0,\infty)$, it has the special decomposition $\widehat{X} = \widehat{M} - \widehat{A}$, where $\widehat{M} = -X$ and $\widehat{A} = \int_0 \gamma(s) \lambda(s) \oo_{\{s < \rho\}} \dd s$. We have $\P(A_\infty = \infty) > 0$, and thus, by Corollary~\ref{C:conv2} we see that the integrability condition in Corollary~\ref{C:conv001}\ref{C:conv001:e} is non-redundant. This example also illustrates that~\eqref{eq:XMA} does not hold in general.
\qed
\end{example}

\begin{example}\label{ex:6.8}
Also in the case where $X$ is quasi-left continuous, the uniform integrability assumption in Corollary~\ref{C:convYX2} cannot be weakened to $L^1$--boundedness. We again put ourselves in the setup of this subsection and suppose $\lambda$ and $\gamma$ satisfy \eqref{A:eq:int1}--\eqref{A:eq:int3}. Then, while $X$ diverges with positive probability, it nonetheless satisfies
\begin{align}  \label{A:eq:prop6.1}
	\sup_{\sigma\in\Tcal} \E\left[ e^{c\log(1+x) * (\mu^X-\nu^X)_\sigma }\right] < \infty
\end{align}
for all $c<1$.
Indeed, if $c \leq 0$, then $$e^{c\log(1+x) * (\mu^X-\nu^X)_\sigma} \leq e^{|c| \log(1+x) * \nu^X_\rho} \leq e^{|c| \int_0^\infty \log(1+s)  \lambda(s) \dd s} < \infty $$ for all $\sigma \in \Tcal$. 
If $c \in (0,1)$, the left-hand side of~\eqref{A:eq:prop6.1} is bounded above by
\begin{align*}
\sup_{\sigma \in \mathcal T} \E\left[e^{c\log(1+x)*\mu^X_\sigma}\right]
&\le 1 + \sup_{\sigma \in \mathcal T} \E\left[(1+\gamma(\rho))^c\,\1{\rho\le\sigma}\right] 
\le 1 + \E\left[(1+x)^c*\mu^X_{\infty}\right] \\
&=1 + \E\left[(1+x)^c*\nu^X_{\infty}\right] \le 1 +  \int_0^\infty(1+\gamma(s))^c\lambda(s)ds < \infty,
\end{align*}
due to \eqref{A:eq:int3}. 
\qed
\end{example}

\subsection{Counterexamples for Novikov-Kazamaki conditions} \label{SS:counter NK}

We now apply the constructions in the previous subsections to construct two examples that illustrate that the uniform integrability assumption in Corollary~\ref{C:NK1}\ref{C:NK1:5} cannot be weakened to $L^1$--boundedness. In the first example we consider predictable---in fact, deterministic---jump times, while in the second example there is one single jump that occurs at a totally inaccessible stopping time.

\begin{example}
Let $(\xi_n)_{n\in\N}$ be a sequence of independent random variables, defined on some probability space $(\Omega,\Fcal,\P)$, such that
\begin{align*}
\P\left( \xi_n = 1\right) &= \frac{1-p_n}{2}; \\
\P\left( \xi_n = -\frac{1-p_n}{1+p_n}\right) &= \frac{1+p_n}{2},
\end{align*}
where $(p_n)_{n\in\N}$ is the sequence from Example~\ref{ex:ui}. Let $M$ be given by $M_t=\sum_{n =1}^{[t]} \xi_n$ for all $t \in \N$, which is a martingale with respect to its natural filtration. Fix $a> 0$. We claim the following: The local martingale $Z=\Ecal(M)$ satisfies $\sup_{\sigma\in\Tcal}\E_\P[e^{B_\sigma^a}]<\infty$ for all $a >0$, but nonetheless fails to be a uniformly integrable martingale.

Let $\Q$ be the F\"ollmer associated with $Z$; see Theorem~\ref{T numeraire}. The process $N$ in~\eqref{eq:N} is then a pure jump martingale under~$\Q$, constant between integer times, with $\Delta N_n=-1/2$ if $\xi_n=1$, and $\Delta N_n=(1-p_n)/(2p_n)$ otherwise for each $n \in \N$. In view of Example~\ref{ex:ui}, the process $N$ explodes under~$\Q$. Hence $Z$ is not a uniformly integrable martingale under~$\P$. However, Lemma~\ref{L:ABdecomp} and~\eqref{ex:ui:eq1} in Example~\ref{ex:ui} yield
\[
\sup_{\sigma\in\Tcal}\E_\P\left[e^{B_\sigma^a}\right] = \sup_{\sigma\in\Tcal}\E_\Q\left[e^{(1-a) \log(1+x) * (\mu^N - \widehat\nu^N)_\sigma}\1{\sigma<\tau_\infty}\right] < \infty,
\]
where $\widehat\nu^N=\nu^N/(1+y)$ is the compensator of $\mu^N$ under~$\Q$.
\qed
\end{example}

\begin{example}
Let $N=X$ be the martingale constructed in Example~\ref{ex:6.8} but now on a probability space $(\Omega,\Fcal,\Q)$. Next, define the process $M$ in accordance with Theorem~\ref{T:reciprocal} as
\[
M = -N + \frac{y^2}{1+y}*\mu^N = -\frac{\gamma(\rho)}{1+\gamma(\rho)}\1{\rho\le t} + \int_0^{t\wedge\rho}\gamma(s)\lambda(s)ds.
\]
Then $M$ is a local martingale under the F\"ollmer measure $\P$ associated with~$\Ecal(N)$. Further, $Z=\Ecal(M)$ cannot be a uniformly integrable martingale under~$\P$, since~$N$ explodes with positive probability under~$\Q$. Nonetheless, thanks to~\eqref{A:eq:prop6.1} we have
\begin{align*}
\sup_{\sigma \in \mathcal T} \E_\P\left[e^{B_\sigma^a}\1{\sigma<\tau_0}\right] 
&=\sup_{\sigma \in \mathcal T}\E_\Q\left[e^{(1-a)\log(1+y)*(\mu^N-\widehat\nu^N)_\sigma}\right]< \infty,
\end{align*}
where again $\widehat\nu^N=\nu^N/(1+y)$ is the compensator of $\mu^N$ under~$\Q$. We conclude that \eqref{eq:4.7} is not enough in general to guarantee that $\Ecal(M)$ be a uniformly integrable martingale.
\qed
\end{example}

\appendix

\section{Stochastic exponentials and logarithms}  \label{A:SE}

In this appendix we discuss stochastic exponentials of semimartingales on stochastic intervals.

\begin{definition}[Maximality]\label{D:maximal}
Let $\tau$ be a foretellable time, and let $X$ be a semimartingale on $\lc0,\tau\lc$. We say that $\tau$ is {\em $X$--maximal} if the inclusion $\{\lim_{t\to\tau}X_t \text{ exists in }\R\}\subset\{\tau=\infty\}$ holds almost surely. \qed
\end{definition}

\begin{definition}[Stochastic exponential]\label{D:stochExp}
Let $\tau$ be a foretellable time, and let $X$ be a semimartingale on $\lc 0,\tau\lc$ such that $\tau$ is $X$--maximal. The \emph{stochastic exponential of $X$} is the process $\mathcal E(X)$ defined by
\[
\mathcal E(X )_t = \exp\left(X_t - \frac{1}{2}[X^c,X^c]_t \right) \prod_{0<s\le t} (1+\Delta X_s)e^{-\Delta X_s}
\]
for all $t< \tau$, and by $\mathcal E(X)_t=0$ for all $t\ge\tau$. \qed
\end{definition}

If $(\tau_n)_{n\in\N}$ is an announcing sequence for $\tau$, then $\Ecal(X)$ of Definition~\ref{D:stochExp} coincides on $\lc0,\tau_n\lc$ with the usual (Dol\'eans-Dade) stochastic exponential of $X^{\tau_n}$. In particular, the two notions coincide when $\tau=\infty$. Many properties of stochastic exponentials thus remain valid. For instance, if $\Delta X_t=-1$ for some $t\in[0,\tau)$ then $\Ecal(X)$ jumps to zero at time $t$ and stays there. If $\Delta X>-1$ then $\Ecal(X)$ is strictly positive on~$\lc0,\tau\lc$. Also, on $\lc0,\tau\lc$, $\mathcal E(X)$ is the unique solution to the equation
\begin{equation*}
Z = e^{X_0} + Z_- \cdot X \quad \text{on} \quad \lc 0,\tau\lc;
\end{equation*}
 see  \citet{Doleans_1976}. 
It follows that $\mathcal E(X)$ is a local martingale on $\lc 0,\tau\lc$ if and only if $X$ is. We record the more succinct expression
\begin{align*}
\mathcal E( X ) = \oo_{\lc0,\tau_0\lc} \exp\left(X - \frac{1}{2}[X^c,X^c] - (x - \log(1+x)) * \mu^X\right),
\end{align*}
where $\tau_0=\tau\wedge\inf\{t \geq 0:\Delta X_t=-1\}$. If $X$ is a local supermartingale on $\lc0,\tau\lc$ with $\Delta X \geq -1$, Theorem~\ref{T:conv} in conjunction with the $X$--maximality of $\tau$ shows that $\lim_{t\to\tau}\Ecal(X)_t=0$ almost surely on $\{\tau<\infty\}$.

We now consider the stochastic logarithm of a nonnegative semimartingale $Z$ that stays at zero after reaching it. In preparation for this, recall that for a stopping time $\rho$ and a set $A\in\mathcal F$, the \emph{restriction of $\rho$ to $A$} is given by
\[
\rho(A) = \rho \oo_A + \infty \oo_{A^c}.
\]
Here $\rho(A)$ is a stopping time if and only if $A\in\mathcal F_\rho$. Define the following stopping times associated to a nonnegative semimartingale $Z$ (recall our convention that $Z_{0-}=Z_0$):
\begin{align}
\nonumber
\tau_0 &= \inf\{ t\ge 0: Z_t = 0\};\\
\label{eq:tauC}
\tau_c &= \tau_0(A_C), \quad\quad A_C =\{Z_{\tau_0-}=0\};\\
\tau_J &= \tau_0(A_J), \quad\quad A_J = \{Z_{\tau_0-}>0\}. \nonumber
\end{align}
These stopping times correspond to the two ways in which $Z$ can reach zero: either continuously or by a jump. We have the following property of $\tau_c$.

\begin{lemma}\label{L:tauCpred}
Fix some nonnegative semimartingale $Z$.
The stopping time $\tau_c$ of \eqref{eq:tauC} is foretellable.
\end{lemma}

\begin{proof}
We must exhibit an announcing sequence for $\tau_c$, and claim that $(\sigma_n)_{n\in\mathbb N}$ is such a sequence, where
\begin{align*}
	\sigma_n = n\wedge\sigma'_n(A_n), \qquad \sigma'_n = \inf\left\{t\ge 0 : Z_t  < \frac{1}{n} \right\}\wedge n, \qquad A_n = \{Z_{\sigma_n'}>0\}.
\end{align*}
To prove this, we first observe that $\sigma_n = n < \infty = \tau_c$ on $A_n^c$ for all $n \in \N$. Moreover, we have $\sigma_n \leq \sigma_n'< \tau_c$ on $A_n$ for all $n \in \N$, where we used that $Z_{\tau_c} = 0$ on the event $\{\tau_c < \infty\}$. We need to show that $\lim_{n \to \infty} \sigma_n = \tau_c$.  On the event $A_C$, see \eqref{eq:tauC}, we have $\tau_c = \tau_0 = \lim_{n \to \infty} \sigma_n' = \lim_{n \to \infty} \sigma_n$ since $A_C \subset A_n$ for all $n \in \N$. On the event $A_C^c = \bigcup_{n =1}^\infty A_n^c$, we have $\tau_c = \infty = \lim_{n \to \infty} n = \lim_{n \to \infty} \sigma_n$. Hence $(\sigma_n)_{n \in \N}$ is an announcing sequence of $\tau_c$, as claimed.
\end{proof}

If a nonnegative semimartingale $Z$  reaches zero continuously, the process $H = \frac{1}{Z_-}\1{Z_->0}$ explodes in finite time, and is therefore not left-continuous. In fact, it is not $Z$--integrable. However, if we view $Z$ as a semimartingale on the stochastic interval $\lc 0,\tau_c\lc$, then $H$ is $Z$--integrable in the sense of stochastic integration on stochastic intervals. Thus $H \cdot Z$ exists as a semimartingale on~$\lc0,\tau_c\lc$, which we will call the stochastic logarithm of $Z$.

\begin{definition}[Stochastic logarithm]\label{D:stochLog}
Let $Z$ be a nonnegative semimartingale with $Z=Z^{\tau_0}$.
The  semimartingale $\mathcal L(Z)$ on $\lc0,\tau_c\lc$ defined by
\begin{align*}
\mathcal L( Z ) = \frac{1}{Z_-}\1{Z_->0} \cdot Z \quad \text{on} \quad \lc0,\tau_c\lc
\end{align*}
is called the \emph{stochastic logarithm of $Z$}.\qed
\end{definition}

We now clarify the relationship of stochastic exponentials and logarithms in the local martingale case. To this end, let $\mathfrak{Z}$ be the set of all nonnegative local martingales $Z$ with $Z_0 = 1$. Any such process~$Z$ automatically satisfies $Z = Z^{\tau_0}$. Furthermore, let $\mathfrak L$ denote the set of all stochastic processes~$X$ satisfying the following conditions:
\begin{enumerate}
\item $X$ is a local martingale on $\lc0,\tau\lc$ for some foretellable, $X$--maximal time $\tau$.
\item $X_0=0$, $\Delta X\ge -1$ on $\lc0,\tau\lc$, and $X$ is constant after the first time $\Delta X=-1$.
\end{enumerate}

The next theorem extends the classical correspondence between strictly positive local martingales and local martingales with jumps strictly greater than~$-1$. The reader is referred to Proposition~I.5 in \citet{Lepingle_Memin_Sur} and Appendix~A of \citet{K_balance} for related results.

\begin{theorem}[Relationship of stochastic exponential and logarithm] \label{T:SE}
The stochastic exponential $\mathcal{E}$ is a bijection from $\mathfrak{L}$ to $\mathfrak{Z}$, and its inverse is the stochastic logarithm $\mathcal L$.  Consider $Z = \mathcal{E}(X)$ for some $Z\in\mathfrak Z$ and $X \in \mathfrak{L}$. The identity $\tau = \tau_c$ holds almost surely, where $\tau$ is the foretellable $X$--maximal time corresponding to $X$, and $\tau_c$ is given by~\eqref{eq:tauC}.
\end{theorem}

\begin{proof}
First, $\Ecal$ maps each $X\in\mathfrak L$ to some $Z\in\mathfrak Z$, and the corresponding $X$--maximal foretellable time~$\tau$ equals $\tau_c$. To see this, note that the restriction of $Z = \mathcal{E}(X)$ to $\lc0,\tau\lc$ is a nonnegative local martingale on $\lc0,\tau\lc$. By Lemma~\ref{L:SMC}, it can be extended to a local martingale~$\overline{Z}$. The implication \ref{T:conv:a} $\Longrightarrow$ \ref{T:conv:g} in Theorem~\ref{T:conv} yields $Z = \overline{Z}$, whence $Z \in \mathfrak{Z}$. We also get $\tau_c=\tau$.

Next, $Z=\mathcal{E}(\mathcal{L}(Z))$ for each $Z\in\mathfrak Z$. This follows from $Z = 1 + Z_-(Z_-)^{-1} \1{Z_->0} \cdot Z = 1 + Z_-  \cdot \mathcal{L}(Z)$ in conjunction with uniqueness of solutions to this equation.

Finally, $\Lcal$ maps each $Z\in\mathfrak Z$ to some $X\in\mathfrak L$, and $\tau_0$ is equal to the corresponding $X$--maximal foretellable time~$\tau$. Indeed, $X = \mathcal{L}(Z)$ is a local martingale on $\lc0, \tau_c\lc$ with jumps $\Delta X=\Delta Z / Z_-\ge-1$, and is constant after the first time $\Delta X=-1$. Moreover, since $Z=\Ecal(\Lcal(Z))=\Ecal(X)$, the implication \ref{T:conv:g} $\Longrightarrow$ \ref{T:conv:a}  in Theorem~\ref{T:conv} yields that $\tau_c$ is $X$--maximal. Thus $X\in\mathfrak L$, with $\tau_c$ being the corresponding $X$--maximal foretellable time.
\end{proof}

Reciprocals of stochastic exponentials appear naturally in connection with changes of probability measure. We now develop some identities relating to such reciprocals. The following function plays an important role:
\begin{equation} \label{eq:phi}
\phi: (-1,\infty) \to (-1,\infty), \qquad \phi(x) = -1 + \frac{1}{1+x}.
\end{equation}
Note that $\phi$ is an involution, that is, $\phi(\phi(x))=x$. The following notation is convenient: Given functions $F:\Omega\times\R_+\times\R\to\R$ and $f:\R\to\R$, we write $F\circ f$ for the function $(\omega,t,x)\mapsto F(\omega,t,f(x))$. We now identify the reciprocal of a stochastic exponential or, more precisely, the stochastic logarithm of this reciprocal. Part of the following result is contained in Lemma~3.4 of~\citet{KK}.

\begin{theorem}[Reciprocal of a stochastic exponential] \label{T:reciprocal}
Let $\tau$ be a foretellable time, and let $M$ be a local martingale on $\lc 0,\tau\lc$ such that $\Delta M> -1$. Define a semimartingale $N$ on $\lc0,\tau\lc$ by
\begin{align}  \label{eq:defN}
	N = -M + [M^c, M^c] + \frac{x^2}{1+x} * \mu^M.
\end{align}
Then $\mathcal{E}(M) \mathcal{E}(N) = 1 \quad\text{on}\quad\lc0,\tau\lc$. Furthermore, a predictable function $F$ is $\mu^M$--integrable ($\nu^M$--integrable) if and only if $F \circ \phi$ is $\mu^N$--integrable ($\nu^N$--integrable). In this case, we have 
	\begin{align}  \label{eq:fmu}
		F*\mu^M = (F \circ \phi) * \mu^N
	\end{align}
on $\lc0,\tau\lc$.
The same formula holds if $\mu^M$ and $\mu^N$ are replaced by $\nu^M$ and $\nu^N$, respectively.
\end{theorem}

\begin{remark}
	Since $|x^2/(1+x)| \leq 2x^2$ for $|x|\le1/2$, the process $x^2/(1+x) * \mu^M$ appearing in \eqref{eq:defN} is finite-valued on $\lc0,\tau\lc$.\qed
\end{remark}

\begin{remark}\label{R:alternative}
	Since $\phi$ is an involution, the identity \eqref{eq:fmu} is equivalent to
	\begin{align}  \label{eq:gmu}
		(G\circ \phi) * \mu^M = G * \mu^N,
	\end{align}
where $G$ is a $\mu^N$--integrable function. The analogous statement holds for $\nu^M$ and $\nu^N$ instead of $\mu^M$ and $\mu^N$, respectively. \qed
\end{remark}

\begin{proof}[Proof of Theorem~\ref{T:reciprocal}]
Note that we have
\[
\Delta N = -\Delta M + \frac{(\Delta M)^2}{1+\Delta M} = \phi(\Delta M) \quad \text{on}\quad \lc0, \tau\lc.
\]
This implies the equality $G*\mu^N=(G\circ\phi)*\mu^M$ on $\lc0,\tau\lc$ for every nonnegative predictable function~$G$. Taking $G=F\circ \phi$ and using that $\phi$ is an involution yields the integrability claim as well as~\eqref{eq:fmu}, and hence also~\eqref{eq:gmu}. The corresponding assertion for the predictable compensators follows immediately. Now, applying~\eqref{eq:gmu} to the function $G(y) = y-\log(1+y)$ yields
\[
(y-\log(1+y))*\mu^N = \left(-1+\frac{1}{1+x} + \log(1+x)\right) * \mu^M \quad \text{on}\quad \lc0, \tau\lc. 
\]
A direct calculation then gives $\Ecal(M)=1/\Ecal(N)$ on $\lc0,\tau\lc$. This completes the proof.
\end{proof}

\section{The F\"ollmer measure} \label{S:follmer}

In this appendix we review the construction of a probability measure that only requires the Radon-Nikodym derivative to be a nonnegative local martingale.  We rely on results by \citet{Pa}, \citet{F1972}, and \citet{M}, and refer to \citet{Perkowski_Ruf_2014} and \citet{CFR2011} for further details, generalizations, proofs, and related literature. A concise description of the construction is available in~\citet{Larsson_2013}.

Consider a filtered probability space $(\Omega,\Fcal,\F,\P)$, where $\Omega$ is a set of possibly explosive paths taking values in a Polish space, $\F$ is the right-continuous modification of the canonical filtration, and $\Fcal = \bigvee_{t\geq 0} \Fcal_t$; see Assumption~$(\mathcal{P})$ in \citet{Perkowski_Ruf_2014} for details. Let $Z$ denote a nonnegative local martingale with $Z_0 = 1$, and define the stopping times $\tau_0=\inf\{t \geq 0:Z_t=0\}$ and $\tau_\infty=\lim_{n\to\infty}\inf\{t\geq 0:Z_t\ge n\}$. Assume that $Z$ does not jump to zero, that is, $Z_{\tau_0-}=0$ on $\{\tau_0<\infty\}$.  For notational convenience we assume, without loss of generality, that we work with a version of $Z$ that satisfies $Z_t(\omega) = \infty$ for all $(t,\omega)$ with $\tau_\infty(\omega) \leq t$.

\begin{theorem}[F\"ollmer's change of measure]  \label{T numeraire}
Under the assumptions of this appendix, there exists a probability measure $\Q$ on $\mathcal F$, unique on $\mathcal F_{\tau_\infty-}$, such that
\[
\E_\P\left[ Z_\sigma G\right] = \E_\Q\left[ G\1{Z_\sigma<\infty}\right]
\]
holds for any stopping time $\sigma$ and any nonnegative $\Fcal_\sigma$--measurable random variable~$G$. Moreover, $ Y = (1/Z)\oo_{\lc0,\tau_\infty\lc}$ is a nonnegative $\Q$--local martingale that does not jump to zero. Finally, $Z$ is a uniformly integrable martingale under~$\P$ if and only if $\Q(Y_{\infty-}>0)=1$, that is, if and only if $Z$ does not explode under~$\Q$.
\end{theorem}

\begin{proof}
The statement is proven in Propositions~2.3 and 2.5 and in Theorem~3.1 in  \citet{Perkowski_Ruf_2014}, after a  change of time that maps $[0,\infty]$ to a compact time interval; see also Theorem~2.1 in \citet{CFR2011}. 
\end{proof}

Now, let $M=\Lcal(Z)$ be the stochastic logarithm of $Z$. Thus $Z=\Ecal(M)$, and $M$ is a $\P$--local martingale on $\lc0,\tau_0\lc$ with $\Delta M>-1$. The following lemma identifies the compensator under~$\Q$ of the jump measure of~$N$, defined in \eqref{eq:defN}. While it can be obtained using general theory (e.g., Theorem~III.3.17 in \citet{JacodS}), we give an elementary proof for completeness.

\begin{lemma} \label{L:predc}
Under the assumptions of this appendix, let $\Q$ denote the probability measure in Theorem~\ref{T numeraire} corresponding to $Z$. Then the compensator under $\Q$ of $\mu^N$ is given by $\widehat{\nu}^N =  \nu^N/(1+y)$.
\end{lemma}

\begin{proof}
Let $G$ be a nonnegative predictable function and define $F=G\circ\phi$. By monotone convergence and Thereom~II.1.8 in \citet{JacodS}, the claim is a simple consequence of the equalities
\[
\E_\Q\left[ G * \mu^N_\sigma\right] = \E_\P\left[\mathcal{E}(M)_\sigma \, \left(F * \mu^M_\sigma\right)\right]  = \E_\P\left[\mathcal{E}(M)_\sigma \, \left((1+x)F* \nu^M_\sigma\right)\right] = \E_\Q\left[ G  * \frac{\nu^N_\sigma}{1+y}\right],
\]
valid for any stopping time $\sigma<\tau_0\wedge\tau_\infty$ ($(\P+\Q)/2$--almost surely) such that $\Ecal(M)^\sigma$ is a uniformly integrable martingale under~$\P$. The first equality follows from Theorem~\ref{T:reciprocal} and Theorem~\ref{T numeraire}, as does the third equality. 

We now prove the second equality. The integration-by-parts formula, the equality $[\mathcal{E}(M),F * \mu^M]=(x\mathcal{E}(M)_-F) * \mu^M$, and the associativity rule $\mathcal{E}(M)_- \cdot (F * \mu^M)=(\mathcal{E}(M)_-F) * \mu^M$ yield
\begin{equation}
\mathcal{E}(M) (F * \mu^M)= (\mathcal{E}(M)_-(1+x)F) * \mu^M + (F * \mu^M)_- \cdot \mathcal{E}(M)\label{eq:L3a}
\end{equation}
on $\lc0,\tau_0\lc$, and similarly 
\begin{align}
\mathcal{E}(M) ((1+x)F * \nu^M) &= (\mathcal{E}(M)_-(1+x)F) * \nu^M+ ((1+x)F * \nu^M)_- \cdot \mathcal{E}(M)   \nonumber  \\
&\qquad+ [\mathcal{E}(M),(1+x) F * \nu^M].   \label{eq:L3b}
\end{align}
Let $(\tau_n)_{n\in\N}$ be a localizing sequence for the following three $\P$--local martingales on $\lc0,\tau_0\lc$: $(F * \mu^M)_- \cdot \mathcal{E}(M)$,  $((1+x) F * \nu^M)_- \cdot \mathcal{E}(M)$, and $[\mathcal{E}(M),(1+x) F * \nu^M]$. The latter is a local martingale on $\lc0,\tau_0\lc$ by Yoeurp's lemma; see Example~9.4(1) in \citet{HeWangYan}. Taking expectations in \eqref{eq:L3a} and \eqref{eq:L3b} and recalling the defining property of the predictable compensator yields
\[
\E_\P\left[\mathcal{E}(M)_\sigma (F * \mu^M)_{\sigma\wedge\tau_n}\right]=\E_\P\left[\mathcal{E}(M)_\sigma ((1+x)F * \nu^M)_{\sigma\wedge\tau_n}\right].
\]
Monotone convergence gives the desired conclusion.
\end{proof}

\section{Extended local integrability under a change of measure} \label{App:Z}

Let $Z$ be a nonnegative local martingale with explosion time $\tau_\infty=\lim_{n\to\infty}\inf\{t\geq 0:Z_t\ge n\}$, and let~$\Q$ be the corresponding F\"ollmer measure; see Theorem~\ref{T numeraire}. It was mentioned in Remark~\ref{R:ELI Q P} that local uniform integrability under~$\Q$ can be characterized in terms of~$\P$. We now provide this characterization.

\begin{definition}[Extended $Z$--localization]  \label{D:ELZI}
With the notation of this appendix, a  progressive process $U$ is called {\em extended locally $Z$--integrable} if there exists a nondecreasing sequence  $(\rho_m)_{m \in \N}$ of stopping times  such that the following two conditions hold:
\begin{enumerate}
\item\label{D:ELZI:i}  $\lim_{m \to \infty} \E_\P[\1{\rho_m < \infty}Z_{\rho_m}] = 0$.
\item\label{D:ELZI:ii}   $\lim_{n \to \infty} \E_\P[\sup_{t \leq \tau_n \wedge \rho_m} U_t\, Z_{\tau_n \wedge \rho_m}] < \infty$ for all $m \in \N$ and a localizing sequence $(\tau_n)_{n \in \N}$ of $Z$.
\end{enumerate}
\end{definition}

\begin{remark} \label{R:5.7}
We make the following remarks concerning Definition~\ref{D:ELZI}.
\begin{itemize}
\item If $Z$ is not a uniformly integrable martingale under $\P$, then a sequence $(\rho_m)_{m\in\N}$ satisfying condition \ref{D:ELZI:i}   of Definition~\ref{D:ELZI} can never be a localizing sequence for $Z$. To see this, note that if~$Z$ is not a uniformly integrable martingale under~$\P$, then
\[
\E_\P[\1{\sigma < \infty}Z_\sigma] \geq \E_\P[Z_\sigma]  -  \E_\P[Z_\infty] = 1  -  \E_\P[Z_\infty]  > 0
\]
for any stopping time $\sigma$ such that $Z^\sigma$ is a uniformly integrable martingale. 
\item As a warning, we note that $U$ being c\`adl\`ag adapted and extended locally bounded is, in general, not sufficient for $U$ being extended locally $Z$--integrable. However, clearly $U$ being bounded is sufficient.
\qed
\end{itemize}
\end{remark}

\begin{lemma} \label{L:extended locally}
With the notation and assumptions of this appendix, let $U$ be a progressive process on $\lc0,\tau_\infty\lc$ with $U_0 = 0$.  Then $U$ is extended locally integrable under $\Q$ if and only if $U$ is extended locally $Z$--integrable under $\P$.
\end{lemma}

\begin{proof}
Let $(\rho_m)_{m \in \N}$ be a nondecreasing sequence of stopping times.  Then $\lim_{m \to \infty} \Q(\rho_m \geq \tau_\infty) = 1$ if and only if $\lim_{m \to \infty} \E_\P[\1{\rho_m < \infty}Z_{\rho_m}] = 0$ since 
\begin{align*}
\Q(\rho_m < \tau_\infty)  =  \E_\Q\left[\1{\rho_m<\tau_\infty} Z_{\rho_m} \frac{1}{Z_{\rho_m}}\right]  =  \E_\P\left[\1{\rho_m<\infty} Z_{\rho_n} \right]
\end{align*}
for each $m \in \N$.

Next, since $U$ is progressive, the left-continuous process $\sup_{s<\cdot}|U_s|$ is adapted to the $\P$--augmentation $\overline\F$ of $\F$; see the proof of Theorem~IV.33 in~\cite{Dellacherie/Meyer:1978}. Hence it is $\overline \F$--predictable, so we can find an $\F$--predictable process $V$ that is indistinguishable from it; see Lemma~7 in Appendix~1 of~\cite{Dellacherie/Meyer:1982}. Setting $W_t=\max(V_t,|U_t|)$ it follows that $W$ is progressive with respect to $\F$ and satisfies $\sup_{s\le\cdot}|U_s|= W$ almost surely. Moreover, for each $m \in \N$,
\begin{align*}
\E_\Q\left[W_{\rho_m}\right]  = \lim_{n \to \infty} \E_\Q\left[W_{\tau_n \wedge \rho_m}  \right]  = \lim_{n \to \infty} \E_\P\left[ W_{\tau_n \wedge \rho_m} Z_{\tau_n \wedge \rho_m}\right]
\end{align*}
for any localizing sequence $(\tau_n)_{n \in \N}$ of $Z$. These observations yield the claimed characterization.
\end{proof}

\section{Embedding into path space} \label{app:embed}

The arguments in Section~\ref{S:NK} relate the martingale property of a nonnegative local $\P$--martingale $Z$ to the  convergence of the stochastic logarithm $N$ of $1/Z$ under a related probability measure $\Q$.  We can only guarantee the existence of this probability measure if our filtered measurable space is of the canonical type discussed in Appendix~\ref{S:follmer}.  We now argue that this is not a restriction.  In a nutshell,  we can always embed $Z = \mathcal E(M)$ along with all other relevant processes into such a canonical space. 

To fix some notation in this appendix, let $(\Omega,\Fcal,\F,\P)$ be a filtered probability space where the filtration $\F$ is right-continuous, let $\tau$ be a foretellable time, and let $M$ be a local martingale on $\lc0,\tau\lc$. Furthermore, let $(H^n)_{n\in\N}$ be an arbitrary, countable collection of c\`adl\`ag adapted processes on $\lc0,\tau\lc$ and let  $\G$ denote the right-continuous modification of the filtration generated by $(M, (H^n)_{n\in \N})$.

Define  $E=\R\times\R\times\cdots$ (countably many copies of $\R$) equipped with the product topology and note that $E$ is Polish. Let $\Delta \notin E$ denote some cemetery state.  Let $\widetilde\Omega$ consist of all  functions $\widetilde\omega:\R_+\to E \cup \{\Delta\}$ that are c\`adl\`ag on $[0,\zeta(\widetilde\omega))$, where $\zeta(\widetilde\omega)=\inf\{t\ge0:\widetilde\omega(t)=\Delta\}$, and that satisfy $\widetilde\omega(t)=\Delta$ for all $t\ge\zeta(\widetilde\omega)$. Let $\widetilde\F=(\widetilde{\Fcal}_t)_{t\ge0}$ be the right-continuous filtration generated by the coordinate process, and set $\widetilde\Fcal=\bigvee_{t\ge0}\widetilde{\Fcal}_t$.  

\begin{theorem}[Embedding into canonical space]\label{T:can_embedding}
	Under the notation of this appendix, there exist a measurable map $\Phi: (\Omega, \Fcal) \rightarrow (\widetilde{\Omega}, \widetilde {\Fcal})$ 
	and c\`adl\`ag  $\widetilde \F$--adapted processes $\widetilde{M}$ and  $(\widetilde H^n)_{n\in\N}$  such that  the following properties hold, where  $\widetilde \P = \P \circ \Phi^{-1}$ denotes the push-forward measure:
\begin{enumerate}
	\item\label{T:can_embedding1}  $\zeta$ is foretellable under $\widetilde\P$ and $\tau=\zeta\circ\Phi$, $\P$--almost surely.
	\item\label{T:can_embedding2}  $H^n=\widetilde H^n\circ\Phi$ on $\lc 0, \tau\lc$, $\P$--almost surely for each $n\in\N$.
		\item\label{T:can_embedding3}  
		 $M=\widetilde M\circ\Phi$ on $\lc 0, \tau\lc$, $\P$--almost surely, and $\widetilde M$ is a local $\widetilde \P$--martingale on $\lc0,\zeta\lc$; we denote the compensator of its jump measure by $\nu^{\widetilde M}$.
\item\label{T:can_embedding4} For any measurable function $f: \R\rightarrow \R$,  $f*\nu^M=(f*\nu^{\widetilde M})\circ\Phi$ on $\lc 0, \tau\lc$, $\P$--almost surely, if one side (and thus the other) is well-defined.
	\item\label{T:can_embedding5} For any $\widetilde \F$--optional process $\widetilde U$, the process $U=\widetilde U\circ\Phi$ is $\F$--optional. In particular, $\sigma=\widetilde\sigma\circ\Phi$ is an $\F$--stopping times for any $\widetilde\F$--stopping time~$\widetilde\sigma$.
\end{enumerate}
\end{theorem}

Assume for the moment that Theorem~\ref{T:can_embedding} has been proved and recall from \citet{Perkowski_Ruf_2014} that $(\widetilde \Omega, \widetilde\Fcal, \widetilde\F, \widetilde\P)$ satisfies the assumptions of Theorem~\ref{T numeraire}.
Now, with an appropriate choice of the sequence $(H^n)_{n \in \N}$, Theorem~\ref{T:can_embedding} allows us, without loss of generality, to assume $(\widetilde \Omega, \widetilde\Fcal, \widetilde\F, \widetilde\P)$ as the underlying filtered space when proving the Novikov-Kazamaki type conditions. To illustrate the procedure, suppose $Z=\Ecal(M)$ is a nonnegative local martingale as in Section~\ref{S:NK}, satisfying Theorem~\ref{T:NK}\ref{T:NK:2} for some optional process $U$ that is extended locally $Z$--integrable (see Appendix~\ref{App:Z}). Without loss of generality we may assume $U$ is nondecreasing. We now apply Theorem~\ref{T:can_embedding}. By choosing the family $(H^n)_{n\in\N}$ appropriately, we can find an $\widetilde\F$--optional process $\widetilde U$ with $U=\widetilde U\circ\Phi$ almost surely, that is extended locally $\widetilde Z$--integrable, where $\widetilde Z=\Ecal(\widetilde M)$. Furthermore, we have
\[
\sup_{\widetilde\sigma \in \widetilde{\mathcal{T}}} \E_{\widetilde\P}\left[\widetilde Z_{\widetilde\sigma} f(\epsilon \widetilde N_{\widetilde\sigma}-\widetilde U_{\widetilde\sigma})\right] \le \sup_{\sigma \in \mathcal{T}} \E_\P\left[Z_\sigma f(\epsilon N_\sigma-U_\sigma)\right] < \infty,
\]
where $\widetilde\Tcal$ is the set of all bounded $\widetilde\F$--stopping times, and $\widetilde N$ is given by~\eqref{eq:N} with $M$ replaced by $\widetilde M$. By Theorem~\ref{T:NK}, the local $\widetilde\F$--martingale $\widetilde Z$ is a uniformly integrable martingale, and thus so is~$Z$. Transferring the reverse implication to the canonical space is done in similar fashion, using also Remark~\ref{R:Uspecs}.

We begin the proof of Theorem~\ref{T:can_embedding}  with a technical lemma.
\begin{lemma}[A canonical sub-filtration] \label{L:shrink}
Under the notation of this appendix, let $(f^n)_{n\in\N}$ be a collection of bounded measurable functions on $\R$, each supported in some compact subset of~$\R\setminus\{0\}$.
Then there exists  a countable family $(K^n)_{n \in \N}$ of $\F$--adapted  c\`adl\`ag  processes,  such that  the following properties hold, with $\Hb$ denoting the right-continuous modification of the filtration generated by $(M,f^n*\nu^{M}, H^n, K^n)_{n\in \N}$:
\begin{enumerate}
	\item\label{L:shrink:i} $\tau$ is foretellable with respect to $\Hb$.
	\item\label{L:shrink:ii}  $M$ is an $\Hb$--local martingale on $\lc0,\tau\lc$.
	\item\label{L:shrink:iii}  $f^n*\nu^{M}$ is indistinguishable from the $\Hb$--predictable compensator of $f^n*\mu^{M}$ for each~$n\in\N$.
\end{enumerate}
\end{lemma}

\begin{proof}
Let $(\tau_m)_{m\in\N}$ be a localizing sequence for $M$ announcing $\tau$. Including the c\`adl\`ag  $\F$--adapted processes $\oo_{\lc\tau_m,\infty\lc}$ in the family $(K^n)_{n\in\N}$ makes  $\tau_m$ an $\Hb$--stopping time for each $m \in \N$, and guarantees \ref{L:shrink:i}  and \ref{L:shrink:ii} . 

Next, fix $n\in\N$. Then, the $\F$--martingale $f^n*\mu^{M}-f^n*\nu^{M}$ is clearly $\Hb$--adapted and hence an $\Hb$--martingale. Thus  \ref{L:shrink:iii} follows if the $\F$--predictable  process $X = f^n*\nu^{M}$ is indistinguishable from an $\Hb$--predictable process. Let $(\sigma_m)_{m\in\N}$ be a sequence of $\F$--predictable times with pairwise disjoint graphs covering the jump times of $X$; see Proposition~I.2.24 in \citet{JacodS}. Since $X$ has bounded jumps, we may define a martingale $J^m$ as the right-continuous modification of $(\E[ \Delta X_{\sigma_m}\1{\sigma_m<\infty}\mid\Fcal_t])_{t \geq 0}$, for each $m \in \N$. Let also $(\rho_{m,k})_{k\in\N}$ be an announcing sequence for $\sigma_m$. We then have
\[
\lim_{k\to \infty} J^m_{\rho_{m,k}}\oo_{\rc\rho_{m,k},\infty\lc} = J^m_{\sigma_m-}\oo_{\lc\sigma_m,\infty\lc} = \Delta X_{\sigma_m}\oo_{\lc\sigma_m,\infty\lc}.
\]
Thus, if we include the processes $J^m$ and $\oo_{\lc\rho_{m,k},\infty\lc}$ for $m,k \in \N$ in the family $(K^n)_{n \in \N}$, then each $\Delta X_{\sigma_m}\oo_{\lc\sigma_m,\infty\lc}$ becomes the almost sure limit of $\Hb$--adapted left-continuous processes and hence $\Hb$--predictable up to indistinguishability.  The decomposition
\[
X=X_-+\sum_{m\in\N}\left(\Delta X_{\sigma_m}\oo_{\lc\sigma_m,\infty\lc} - \Delta X_{\sigma_m}\oo_{\rc\sigma_m,\infty\lc}\right)
\]
then implies that, up to indistinguishability, $X$ is a sum of $\Hb$--predictable processes. Repeating the same construction for each $n\in\N$ yields  \ref{L:shrink:iii}. 
\end{proof}

\begin{proof}[Proof of Theorem~\ref{T:can_embedding}]
Let $(M,f^n*\nu^{M}, H^n, K^n)_{n\in \N}$ and $\Hb$ be as in Lemma~\ref{L:shrink}. There exists an $\Hb$--stopping time~$T$ with $\P(T<\infty) = 0$ such that the paths $(M(\omega),f^n*\nu^{M}(\omega), H^n(\omega), K^n(\omega))_{n\in \N}$ are c\`adl\`ag on $[0,T(\omega) \wedge \tau(\omega) )$ for all $\omega \in \Omega$.  We now check that the measurable map $\Phi: (\Omega, \Fcal) \rightarrow (\widetilde{\Omega}, \widetilde {\Fcal})$ given by
\[
\Phi(\omega)(t) = \left( M_t(\omega),f^n*\nu^M_t(\omega),H^n_t(\omega),K^n_t(\omega)\right)_{n \in \N}\1{t<T(\omega) \wedge \tau(\omega)} + \Delta \1{t\ge T(\omega) \wedge \tau(\omega)},
\]
along with the obvious choice of processes $\widetilde{M}$ and $(\widetilde{H}^n)_{n \in \N}$,
satisfies the conditions of the theorem. The statements in \ref{T:can_embedding1} and~\ref{T:can_embedding2} are clear since $\zeta\circ\Phi=T\wedge\tau$. For \ref{T:can_embedding3}, one uses in addition that $\Phi^{-1}(\widetilde\Fcal_t)\subset\Fcal_t$ for all $t\ge0$ due the $\Fcal_t/\widetilde\Fcal_t$--measurability of~$\Phi$.

We now prove \ref{T:can_embedding4}. For each $n \in \N$, let $\widetilde{F}^n$ denote the coordinate process in the canonical space corresponding to $f^n *\nu^M$.  Then, by Lemma~\ref{L:shrink}, for each $n \in \N$, $\widetilde{F}^n$ is indistinguishable from an $\widetilde {\F}$--predictable process and, due to the definition of $\widetilde \P$, the process $f^n * \mu^{\widetilde{M}} - \widetilde{F}^n$ is a $\widetilde{\P}$--local martingale. Here, $\mu^{\widetilde{M}}$ denotes the jump measure of the c\`adl\`ag process $\widetilde M$. Thus $\widetilde{F}^n$ is indistinguishable from $f^n*\nu^{\widetilde{M}}$, which gives $f^n * \nu^M = (f^n * \nu^{\widetilde M}) \circ \Phi$ on $\lc 0, \tau\lc$, $\P$--almost surely, for each $n \in \N$. Choosing $(f^n)_{n \in \N}$ to be a measure-determining family along with a monotone class argument thus yields \ref{T:can_embedding4}.

For~\ref{T:can_embedding5}, let $\widetilde\sigma$ be an $\widetilde\F$--stopping time. Then $\sigma=\widetilde\sigma\circ\Phi$ is an $\F$-stopping times, since
\[
\{\sigma\le t\}=\Phi^{-1}(\{\widetilde\sigma\le t\})\in\Phi^{-1}(\widetilde\Fcal_t)\subset\Fcal_t
\]
for all $t\ge0$. Thus if $\widetilde U=\oo_{\lc0,\widetilde\sigma\lc}$, then $U=\widetilde U\circ\Phi=\oo_{\lc0,\sigma\lc}$ is $\F$--optional. The result now follows by a monotone class argument.
\end{proof}

\setlength{\bibsep}{0.0pt}
\bibliography{aa_bib}{}
\bibliographystyle{apalike}

\end{document}